\documentclass[a4paper, 11pt]{article}

\usepackage{amssymb}
\usepackage{amssymb}

\usepackage{amsmath,amssymb,amsthm}
\usepackage[latin1]{inputenc}
\usepackage{version,tabularx,multicol}
\usepackage{graphicx,float,psfrag}
\usepackage{stmaryrd}
\usepackage{color}
\usepackage{mathtools}
\mathtoolsset{showonlyrefs}
\usepackage[pdftex]{hyperref}
\usepackage{amsfonts}
\usepackage{mathrsfs}
\usepackage{geometry}
\usepackage{longtable}
\usepackage[numbers,sort&compress]{natbib}

\newcommand{\Pp}{\mathbb P}

\theoremstyle{plain}
\newtheorem{theorem}{Theorem}[section]

\newtheorem{proposition}[theorem]{Proposition}

\newtheorem{lemma}[theorem]{Lemma}

\newtheorem{definition}[theorem]{Definition}

\newtheorem{remark}{Remark}[section]

\newtheorem{example}{Example}[section]
\newtheorem{assumption}{Assumption}[section]

\makeatletter
\@addtoreset{equation}{section}
\makeatother

\def\beqlb{\begin{eqnarray}}\def\eeqlb{\end{eqnarray}}
\def\beqnn{\begin{eqnarray*}}\def\eeqnn{\end{eqnarray*}}

\newcommand{\bcen}{\begin{center}}
	\newcommand{\ecen}{\end{center}}
\newcommand{\bgeqn}{\begin{equation}}
	\newcommand{\edeqn}{\end{equation}}

\begin{document}
	\title{A large-deviation principle for the empirical distribution of a regular branching random walk}
	\author{Shuxiong Zhang, Yaping Zhu}
	\maketitle
	\renewcommand{\thefootnote}{\fnsymbol{footnote}}

\begin{abstract}
Let \(\{Z_n\}_{n\geq0}\) be a supercritical branching
random walk with deterministic rooted
\(b\)-ary tree, \(b\geq2\), and symmetric displacements satisfying
\(\lim_{x\to+\infty}x^{-\alpha}\log\Pp(X>x)=-\lambda \) with \(\alpha,\lambda>0\). Set $\overline Z_n(\cdot):=Z_n(\cdot)/Z_n(\mathbb{R}),~n\geq0.$ For a
finite union \(A\) of intervals, we establish the following full large-deviation principle: for every Borel set
\(\Gamma\subset[0,1]\),
\begin{align*}
  -\inf_{q\in\Gamma^\circ}Q_A(q)
  &\leq
  \liminf_{n\to\infty}n^{-\alpha/2}
  \log\Pp\!\left(\overline Z_n(\sqrt n\,A)\in\Gamma\right)\\
  &\leq
  \limsup_{n\to\infty}n^{-\alpha/2}
  \log\Pp\!\left(\overline Z_n(\sqrt n\,A)\in\Gamma\right)\\
  &\leq
  -\inf_{q\in\overline\Gamma}Q_A(q),
\end{align*}
where $Q_A$ is a good rate function on $[0,1]$. Meanwhile, we obtain large deviation probabilities for $\{\overline Z_n(\sqrt n\,A)\}_{n\geq1}.$ This strengthens the nonmatching upper and lower
bounds obtained by Chen and He  [Probab. Theory
Related Fields 175 (2019) 255-307]  for regular trees with Weibull
displacements. Our method combines a large-deviation principle for rescaled
displacement tree fields and exponential equivalence.
\end{abstract}
	
\maketitle

\section{Introduction and main results}
\label{sec:introduction}

\subsection{Introduction}
The branching random walk (BRW) \(\{Z_n\}_{n\geq0}\) is a
measure-valued Markov process whose law is determined by an offspring
distribution \(\{p_k\}_{k\geq0}\) and an \(\mathbb{R}\)-valued random variable
\(X\) representing the step size or displacement, which is described formally below.
\par
At generation \(0\), there is one particle located at the origin, so that
\(Z_0=\delta_0\). At generation \(1\), this particle first produces \(N\)
offspring, where $N$ is a random variable with distribution \(\{p_k\}_{k\geq0}\).
Label these children by \(1,\ldots,N\). Child \(i\) then makes a
displacement \(X_i\), where \((X_i)_{i\geq1}\) are independent copies of
\(X\) and are independent of \(N\). This forms the point process
\[
  Z_1:=\sum_{i=1}^{N}\delta_{X_i}.
\]
At every subsequent generation, all particles alive in the preceding
generation repeat this procedure independently from their current positions:
each particle first produces its offspring, and then each child makes an
independent displacement from the position of its parent.

Let \(\mathbb T\) be a
Galton-Watson tree with offspring distribution
\(\{p_k\}_{k\geq0}\) embedded in the Ulam--Harris tree
\(\bigcup_{n\geq0}{\mathbb{N}}^n\), with root \(\varnothing\). For \(u,v\in\mathbb T\), write \(v\prec u\) if
\(v\) is a strict ancestor of \(u\), and let \(|u|\) denote the generation
of \(u\).  Write \(v\preceq u\) if \(v\prec u\) or $v=u$. The variables \(\{X_v:v\in\mathbb T\setminus
\{\varnothing\}\}\) are independent copies of \(X\), and independent of the
genealogical tree. The position of a particle \(u\in\mathbb T\) is
\[
  S_\varnothing:=0,
  \qquad
  S_u:=\sum_{\varnothing\prec v\preceq u}X_v,
  \qquad u\neq\varnothing.
\]
Then, for $n\geq0$,
\[
  Z_n=\sum_{\substack{u\in\mathbb T\\|u|=n}}\delta_{S_u}.
\]
We refer to
Shi~\cite{Shi2015} for a more detailed overview of branching random walks.
\par
Assume $p_0=0$. For $n\geq0$, let
$$\overline Z_n(\cdot)=\frac{Z_n(\cdot)}{Z_n(\mathbb{R})}$$
be the corresponding empirical distribution of $Z_n$.
Biggins~\cite{Biggins1990} proved that if \(\mathbb{E}[X]=0\), \(\mathbb{E}[X^2]=\sigma^2\), $p_0=0$, $p_1<1$ and $\sum_{k\geq1}(k\log k)p_k<\infty$, then for any $A=(-\infty,x],~x\in\mathbb{R}$,
$$
\lim_{n\to\infty}\overline Z_n(\sqrt {n\sigma^2}A)=\nu(A),~\text{a.s.},
$$
where $\nu$ is the standard Gaussian measure. Later, the almost sure convergence rates were obtained by Chen~\cite{Chen2001} and Gao and Liu \cite{GaoLiu2016}.
\par
The natural large-deviation question is therefore to determine the decay rate of
\[
  \mathbb{P}\!\left(\overline Z_n(\sqrt n\,A)\geq p\right),
  \qquad p\in(\nu(A),1).
  \label{eq:intro-event}
\]
Louidor and Perkins~\cite{LouidorPerkins2015} initiated the systematic study
of such empirical-distribution deviations for branching random walks.  If
$X$ is a symmetric two-point random variable and $p_0=p_1=0$, they identified double-exponential decay and distinguished shift
and dilation strategies.  Later, Louidor and Tsairi~\cite{LouidorTsairi2017}
extended this problem to general branching random walks with point-process reproduction laws with bounded spatial support.
\par
Afterwards, Chen and He~\cite{ChenHe2019} treated unbounded displacements and showed that
the tail of the motion interacts decisively with the offspring law. They
obtained results both in the Schr\"oder regime \(p_1>0\) and in the
B\"ottcher regime \(p_0=p_1=0\), for Weibull and other displacement tails. In the B\"ottcher--Weibull setting, let
\[
  b:=\min\{k\geq0:p_k>0\},
  \qquad
  B:=\sup\{k\geq0:p_k>0\}.
\]
They \cite[Theorem~1.3]{ChenHe2019} gave exact large deviation probabilities when \(B>b\), exploiting
offspring randomness so that the descendants of a selected particle can
dominate the terminal population.  However, when \(B=b\), this strategy is unavailable. In that case, \cite[Theorem~1.3(3)]{ChenHe2019} states that if
\begin{align}\label{4he6hyh62}
I_A(p):=\inf\{|x|:\nu(A-x)\geq p\}
\end{align}
is finite and $I_A$ is continuous at $p$ then there exist constants
\(0<c_{\alpha,p,A}<C_{\alpha,p,A}<\infty\) such that
\begin{align}
  -C_{\alpha,p,A}
  &\leq
  \liminf_{n\to\infty}n^{-\alpha/2}
  \log\mathbb{P}\!\left(\overline Z_n(\sqrt n\,A)\geq p\right)
  \notag\\
  &\leq
  \limsup_{n\to\infty}n^{-\alpha/2}
  \log\mathbb{P}\!\left(\overline Z_n(\sqrt n\,A)\geq p\right)
  \leq-c_{\alpha,p,A}.
  \label{eq:CH-regular-bounds}
\end{align}
More recently, Zhang~\cite{Zhang2022} considered
heavy-tailed offspring or displacements, where a single exceptional family
or jump may dominate the event. Above works \cite{ChenHe2019,LouidorPerkins2015,LouidorTsairi2017,Zhang2022} always deal with large-deviation probabilities rather than the large-deviation principle; see Definition \ref{def:LDP} below for its definition.
\par

The present paper aims to address the $b$-regular tree gap (i.e. $B=b$) in \eqref{eq:CH-regular-bounds} and give a large-deviation principle for this problem. The main contributions are as follows.
\begin{enumerate}
  \item We prove a full large-deviation principle (LDP) for the rescaled displacement fields in a
  separable weighted \(\ell^1\)-space (see Theorem \ref{thm:tree-field-LDP} below). Then, by contraction and exponential
  equivalence, a full LDP for \(\{\overline Z_n(\sqrt n\,A)\}_{n\geq1}\) is established (see Theorem \ref{thm:empirical-LDP} below). We also identify the effective domain and regularity of its rate function.
  \item We convert the large-deviation principle into upper and lower
  bounds of the large-deviation probabilities in \eqref{eq:CH-regular-bounds}, and identify the exact logarithmic constant at every continuity point of the rate function. This refines the results of Chen and He \cite[Theorem 1.3(3)]{ChenHe2019}. Moreover, for the rate function, the continuity criterion, explicit limit constants, concrete continuity classes, and jump examples are given in this paper.
\end{enumerate}

\subsection{Main results}
Before stating the theorem, we make the following assumptions.
\begin{assumption}\label{45twzg5yy43}
In the rest of the paper, we always assume the following:\\
1. There exists some \(b\geq2\) such that $p_b=1$ (i.e. every particle has exactly $b$ children);\\
2. \(X\) is symmetric, \(\mathbb{E}[X]=0\), and \(\mathbb{E}[X^2]=1\);\\
3. There exist \(\alpha,\lambda>0\) such that
  \begin{equation}\label{eq:tail}
    \lim_{x\to\infty}x^{-\alpha}\log\mathbb{P}(X>x)=-\lambda;\\
  \end{equation}
 4. There exists some \(r_A\geq1\) such that the set
  \[
    A=\bigcup_{j=1}^{r_A}I_j,
    \qquad I_j\text{ has endpoints }-\infty\leq a_j<b_j\leq\infty,
  \]
  where the intervals \(I_j,~1\leq j\leq r_A,\) are pairwise disjoint and \(A\notin\{\emptyset,\mathbb{R}\}\). The choice of open or closed endpoints does not affect the function $g_A(\cdot):=\nu(A-\cdot)$, and hence does not affect the rate function $Q_A$ defined in \eqref{eq:QA} below.
\end{assumption}
\par
Set
\[
\begin{aligned}
  m_A:=\inf_{x\in\mathbb{R}}g_A(x),~M_A:=\sup_{x\in\mathbb{R}}g_A(x),~D_A:=\{g_A(x):x\in\mathbb{R}\}.
\end{aligned}
\]
Since $g_A$ is continuous, it is simple to see that
\begin{align}\label{4gr5yhy6yr4}
(m_A,M_A)\subset D_A\subset[m_A,M_A].
\end{align}
 \par
Let $\mathbb{T}_b=\cup_{k\geq0}\{1,2,...,b\}^k$ be a $b$-regular Ulam-Harris tree with root \(\varnothing\). Denote by
\[
  \mathcal{H}_b
  :=
  \left\{
    h=(h_v)_{v\in\mathbb{T}_b}:h_\varnothing=0,\quad
    \|h\|_b:=
    \sum_{k\geq1}b^{-k}\sum_{|v|=k}|h_v|<\infty
  \right\}
\]
the {\em weighted tree space}, where we call $h=(h_v)_{v\in\mathbb{T}_b}$ a {\em tree field } or {\em field } in this paper. Let \(\mu_b\) be the probability measure on
\(\bigl(\{1,\ldots,b\},2^{\{1,\ldots,b\}}\bigr)\) defined by
$
\mu_b(\{i\})=b^{-1},~i\in\{1,\ldots,b\}.
$
Set the infinite product probability space
\begin{align}\label{54gt24g5gt}
(\partial\mathbb T_b,\mathcal F_{\partial},\mathbf m_b)
:=
\Big(
\{1,\ldots,b\}^{\mathbb N^+},
\prod_{k\geq1}2^{\{1,\ldots,b\}},
\prod_{k\geq1}\mu_b
\Big),
\end{align}
where $2^{\{1,\ldots,b\}}$ stands for the $\sigma$-algebra $\{A:A\subset\{1,2,...,b\}\}$ and $\mathbb N^+:=\{1,2,...\}$. For \(\xi\in\partial\mathbb T_b\), denote by $\xi|k$ the particle at generation $k$ in $\xi$. For \(h\in\mathcal{H}_b\), define
\begin{align}\label{5j7u7ugt5}
  H_h(\xi):=\sum_{k\geq1}h_{\xi|k},
\end{align}
which is well-defined for $\mathbf m_b$-a.s. $\xi$; see (\ref{eq:path-L1}) below.

\par
For \(h\in\mathcal{H}_b\), set
\begin{align}\label{5yhyr7u5}
  \Psi_A(h)
  :=
  \int_{\partial\mathbb{T}_b}g_A(H_h(\xi))\,\mathbf m_b(\mathrm{d}\xi),
  \qquad
  \mathcal{I_{\alpha}}(h):=\lambda\sum_{v\neq\varnothing}|h_v|^\alpha,
\end{align}
with \(\mathcal{I_{\alpha}}(h)=\infty\) when the series diverges. Hereafter, we call $\mathcal{I_{\alpha}}(h)$  {\em the tree energy} of $h$. Define the rate function
\begin{equation}\label{eq:QA}
  Q_A(q)
  :=
  \inf\{\mathcal{I_{\alpha}}(h):h\in\mathcal{H}_b,\ \Psi_A(h)=q\},
  \qquad q\in[0,1].
\end{equation}
Here and below, the infimum over the empty set is \(\infty\). We point out that if
\(q\in D_A\),  then $\Psi_A(h)=q$ has uncountably many feasible fields; see \eqref{rsgfdgty5y3} below. If \(q\notin D_A\), there is no feasible field; see Proposition~\ref{prop:QA-regularity} below.
\par
Now, we are ready to present our main theorems.

\begin{theorem}\label{thm:empirical-LDP}
\(\{\overline Z_n(\sqrt n\,A)\}_{n\geq1}\) satisfies a full LDP on \([0,1]\), with speed \(n^{\alpha/2}\) and good rate function \(Q_A\). More precisely, for every Borel set
\(\Gamma\subset[0,1]\),
\begin{align*}
  -\inf_{q\in\Gamma^\circ}Q_A(q)
  &\leq
  \liminf_{n\to\infty}n^{-\alpha/2}
  \log\mathbb{P}\!\left(\overline Z_n(\sqrt n\,A)\in\Gamma\right)\\
  &\leq
  \limsup_{n\to\infty}n^{-\alpha/2}
  \log\mathbb{P}\!\left(\overline Z_n(\sqrt n\,A)\in\Gamma\right)\\
  &\leq
  -\inf_{q\in\overline\Gamma}Q_A(q),
\end{align*}
where $\Gamma^\circ$ ($\overline\Gamma$) stands for the interior (closure) of $\Gamma$. Whenever \(Q_A(q)<\infty\), the infimum in \eqref{eq:QA} is attained.
\end{theorem}
\begin{remark}The rate function $Q_A$ is non-negative and $\{p\in[0,1]: Q_A(p)<\infty\}=\{\nu(A-x):x\in\mathbb{R}\}$.  Moreover, $Q_A(\nu(A))=0$ and is strictly decreasing on $(m_A,\nu(A)]$ and strictly increasing on $[\nu(A),M_A)$; for more properties of $Q_A$ see Proposition \ref{prop:QA-regularity} below.
\end{remark}
For \(p\in(\nu(A),M_A)\), since $Q_A$ is strictly increasing on $[\nu(A),M_A)$ (see Proposition \ref{prop:QA-regularity} (ii)), write
\begin{equation}
  Q_A(p+):=\lim_{q\to p+}Q_A(q).
\end{equation}
Next, we present the large deviation probabilities of the empirical distribution.

\begin{theorem}
\label{thm:main-one-sided}
For every $p\in(\nu(A),M_A)$, we have
\begin{align}
  -Q_A(p+)
  &\leq
  \liminf_{n\to\infty}n^{-\alpha/2}
  \log\mathbb{P}\!\left(\overline Z_n(\sqrt n\,A)\geq p\right)
  \notag\\
  &\leq
  \limsup_{n\to\infty}n^{-\alpha/2}
  \log\mathbb{P}\!\left(\overline Z_n(\sqrt n\,A)\geq p\right)
  \leq -Q_A(p).\nonumber
\end{align}
Consequently, if \(Q_A(p+)=Q_A(p)\), then
\[
  \lim_{n\to\infty}n^{-\alpha/2}
  \log\mathbb{P}\!\left(\overline Z_n(\sqrt n\,A)\geq p\right)
  =-Q_A(p).
\]
\end{theorem}

\begin{remark}[Continuity and possible jumps of $Q_A(p)$]
\label{rem:intro-continuity}
Example~\ref{cor:concrete-continuity} gives three elementary classes of sets $A$:
a single interval, the complement of a bounded interval, and two symmetric
disjoint intervals of equal length. For each class,
\(Q_A(p+)=Q_A(p)\) for every $p\in(\nu(A),M_A)$, so
Theorem~\ref{thm:main-one-sided} gives an exact logarithmic asymptotic. For general finite unions of intervals,
Theorem~\ref{thm:optimizer-escape} gives necessary and sufficient criteria for \(Q_A(p+)=Q_A(p)\),
while Proposition~\ref{cor:no-suboptimal-maxima} gives a simpler sufficient
condition. We also show that there exists a set $A$ such that \(Q_A(p)\neq Q_A(p+)\); see Example~\ref{prop:genuine-jump} below.
\end{remark}
The definition of the rate function $Q_A$ (see \eqref{eq:QA}) may seem abstract. However, under certain conditions, it has explicit forms; see Theorem~\ref{thm:dimension-reduction} below. For convenience, we restate it as follows.
\begin{remark}
Suppose that, for some \(\sigma\in\{-1,1\}\), the map
\(t\mapsto g_A(\sigma t)\) is strictly increasing on
 \(\mathbb{R}\). Fix $p\in(\nu(A),M_A)$. Then \(Q_A(p+)=Q_A(p)\), and the following assertions hold.
\begin{enumerate}
\item If \(0<\alpha\le1\), then
\begin{equation}
  Q_A(p)
  =\lambda\min\left\{
      \sum_{i=1}^b t_i^\alpha:
      t_i\ge0,\quad
      \frac1b\sum_{i=1}^b g_A(\sigma t_i)\ge p
    \right\}.
\end{equation}
This constraint problem can be solved under certain conditions; see Example \ref{5tre6yy6yu}.

\item If $\alpha\geq1$ and \(t\mapsto g_A(\sigma t)\) is concave on
\([0,\infty)\), then
\[
  Q_A(p)=
  \begin{cases}
    \lambda b I_A(p), & \alpha=1,\\[3pt]
    \lambda\bigl(b^{1/(\alpha-1)}-1\bigr)^{\alpha-1}I_A(p)^\alpha,
      & \alpha>1.
  \end{cases}
\]
\end{enumerate}

\end{remark}

The remainder of the paper is organized as follows. Section \ref{sec:preliminaries}
develops the analytic framework needed for proving the main theorems.
In Section \ref{sec:proofs}, by using an exponentially equivalent representation of $\{\overline Z_n(\sqrt n A)\}_{n\geq1}$,
we establish the LDP through a variational representation of the tree energy, and Theorem \ref{thm:empirical-LDP} is proved there.
Theorem \ref{thm:main-one-sided} is a direct consequence of Theorem \ref{thm:empirical-LDP} and the monotonicity of \(Q_A\).
Section \ref{5yehye6t5t} is devoted to giving several criteria for determining the continuity of $Q_A$. Section \ref{rgrthrt45vg} discusses
the resulting strategies and possible extensions.

\section{Preliminaries}\label{sec:preliminaries}
To prove the main results, we need a state space in which
the tree field can be approximated by finitely many generations.
To this end, we introduce the finite-generation truncations in subsection \ref{4berhhyh6r}. In subsection
\ref{rfvgrte6yr4}, we establish the continuity of \(\Psi_A\)
and the compactness of sublevel sets of \(\mathcal{I_{\alpha}}\). Then, we use these
properties to obtain bounds, lower semicontinuity, and attainment for \(Q_A\).
Finally, we collect standard large-deviation results in subsection \ref{4gtgetg4r}.

\subsection{The finite-generation
truncation}\label{4berhhyh6r}

For every integer \(M\geq1\), define the truncation \(P_M:\mathcal{H}_b\to\mathcal{H}_b\) by
\[
  (P_Mh)_v:=h_v\mathbf{1}_{\{|v|\leq M\}}, {v\in\mathbb{T}_b}
\]
and the finite-dimensional subspace
\begin{equation}\label{eq:EM}
  E_M:=P_M\mathcal{H}_b
  =
  \{h\in\mathcal{H}_b:h_v=0\text{ whenever }|v|>M\}
\end{equation}
of \((\mathcal{H}_b,\|\cdot\|_b)\). As usual, the operator norm is defined as $\|P_M\|:=\sup_{\|h\|_{b}\neq 0}\frac{\|P_Mh\|_b}{\|h\|_b}.$
\par
\begin{proposition}\label{prop:Hb-Banach}
The normed space \((\mathcal{H}_b,\|\cdot\|_b)\) is a separable Banach space.
Moreover, \(\|P_M\|\leq1\) and \(\lim_{M\to\infty}P_Mh=h\) in \(\mathcal{H}_b\) for every
\(h\in\mathcal{H}_b\).
\end{proposition}

\begin{proof}
The vertex set
$
  \mathbb{T}_b^\circ:=\mathbb{T}_b\setminus\{\varnothing\}=\cup_{k\geq1}\{1,2,...,b\}^k
$
is countable. Set the space
\[
  \ell^1(\mathbb{T}_b^\circ)
  :=\Big\{y=(y_v)_{v\in\mathbb{T}_b^\circ}:
    \|y\|_{\ell^1(\mathbb{T}_b^\circ)}:=\sum_{v\in\mathbb{T}_b^\circ}|y_v|<\infty\Big\},
\]
and the operator
\[
  T:\mathcal{H}_b\longrightarrow\ell^1(\mathbb{T}_b^\circ),
  \qquad
  (Th)_v:=b^{-|v|}h_v.
\]
Then \(T\) is linear and
\[
  \|Th\|_{\ell^1(\mathbb{T}_b^\circ)}
  =
  \sum_{v\in\mathbb{T}_b^\circ}b^{-|v|}|h_v|
  =
  \|h\|_b.
\]
Conversely, for any \(y\in\ell^1(\mathbb{T}_b^\circ)\), the coordinates
\[
  h_v=b^{|v|}y_v,~v\in\mathbb{T}_b^\circ
\]
define an element \(h\in\mathcal{H}_b\) satisfying \(Th=y\).  Thus \(T\) is a
linear isometric bijection.
\par
It is well known that any countably infinite-dimensional $\ell_1$ space is complete and separable. Since \(\mathbb{T}_b^\circ\) is countable, \(\ell^1(\mathbb{T}_b^\circ)\) is complete and separable. The isometry \(T\) transfers completeness and separability to
\(\mathcal{H}_b\).

Finally, by the definition of $P_M$, we have $\|P_Mh\|_b\leq\|h\|_b$, thus  \(\|P_M\|\leq1\). Since $\|h\|_b=\sum_{k\geq1}b^{-k}\sum_{|v|=k}|h_v|<\infty$, it follows that

\[
  \lim_{M\to\infty}\|h-P_Mh\|_b
  =
   \lim_{M\to\infty}\sum_{k>M}b^{-k}\sum_{|v|=k}|h_v|
  =0.
\]

\end{proof}
\par
We point out that completeness of \((\mathcal{H}_b,\|\cdot\|_b)\) is used in Proposition~\ref{prop:energy-good} to turn closed
and totally bounded tree energy sublevels into compact sets.

\subsection{The regularity of $\Psi_A$ and $\mathcal{I_{\alpha}}$}\label{rfvgrte6yr4}

From Assumption \ref{45twzg5yy43} (4),
\[
  A=\bigcup_{j=1}^{r_A}I_j,
\]
where $I_j$ has endpoints $a_j$ and $b_j$ allowing \(a_j=-\infty\) or \(b_j=\infty\). Then,
\begin{equation}\label{eq:gA-endpoints}
  g_A(x)
  =
  \sum_{j=1}^{r_A}
  \bigl[\Phi(b_j-x)-\Phi(a_j-x)\bigr],
\end{equation}
where \(\Phi(x): =\int^{x}_{-\infty}\frac{1}{\sqrt{2\pi}}e^{-\frac{y^2}{2}}\mathrm{d} y,~x\in\mathbb{R}\). Let \(\varphi(x):=\frac{1}{\sqrt{2\pi}}e^{-\frac{x^2}{2}},~x\in\mathbb{R}\), and
\(\varphi(\pm\infty):=0\).  Hence,
\begin{equation}\label{eq:gA-explicit-Lipschitz}
  |g_A'(x)|
  =
  |\sum_{j=1}^{r_A}
  \bigl[\varphi(a_j-x)-\varphi(b_j-x)\bigr]|
  \leq
  2r_A\sup_{y\in \mathbb{R}}\varphi(y)
  =
  \sqrt{\frac{2}{\pi}}\,r_A,
  \quad x\in \mathbb{R}.
\end{equation}
The mean value theorem yields that for any $x,~y\in\mathbb{R}$,
\begin{equation}\label{eq:gA-Lipschitz}
  |g_A(x)-g_A(y)|
  \leq
  \sqrt{\frac{2}{\pi}}\,r_A|x-y|.
\end{equation}
\par

For \(\xi\in\partial\mathbb T_b\) and \(u=(u_1,\ldots,u_k)\in\{1,\ldots,b\}^k\subset\mathbb T^{\circ}_b\), $k\geq1$, define the
tree cylinder
$$
  [u]
  :=
  \left\{
    \xi=(\xi_1,\xi_2,...)\in\partial\mathbb T_b:
    \xi_j=u_j,\ 1\leq j\leq k
  \right\}
  =
  \left\{
    \xi\in\partial\mathbb T_b:\xi|k=u
  \right\}.
$$
By the definition of the product measure $\mathbf m_b$, it follows that
\begin{align}
  \mathbf m_b([u])=
  \prod_{j=1}^{k}\mu_b(\{u_j\})
  \prod_{j>k}\mu_b(\{1,\ldots,b\})
  =
  \left(\frac{1}{b}\right)^k=
  b^{-|u|}.\nonumber
\end{align}
A particle \(u\) at generation
\(|u|=k\) has exactly \(b^{n-k}\) descendants at generation \(n\), where $n>k$. Therefore,
its descendants always represent the fixed fraction \(b^{-k}\) of the total
population. The identity \(\mathbf m_b([u])=b^{-|u|}\) merely records this
deterministic fraction. Since
\(\mathbf m_b([u])=b^{-|u|}\), it follows that
\begin{align}\label{eq:path-L1}
  \int_{\partial\mathbb{T}_b}
  \sum_{k\geq1}|h_{\xi|k}|\,\mathbf m_b(\mathrm{d}\xi)
  &=
  \sum_{k\geq1}\int_{\partial\mathbb{T}_b}
  |h_{\xi|k}|\,\mathbf m_b(\mathrm{d}\xi)\cr
   &=
  \sum_{k\geq1}\sum_{|v|=k}\int_{[v]}
  |h_{v}|\,\mathbf m_b(\mathrm{d}\xi)\cr
  &=
  \sum_{k\geq1}\sum_{|v|=k}
  |h_v|\mathbf m_b([v])\cr
  &=\sum_{k\geq1}b^{-k}\sum_{|v|=k}|h_v|=\|h\|_b.
\end{align}
Thus, by \eqref{5j7u7ugt5}, \(H_h(\xi)\) is absolutely convergent for
\(\mathbf m_b\)-almost every $\xi$ and
\begin{equation}\label{eq:path-map-bound}
  \|H_h\|_{{\ell}^1(\mathbf m_b)}\leq\|h\|_b.
\end{equation}
\par

The following lemma shows that $\Psi_A$ is Lipschitz continuous. Recall from \eqref{5yhyr7u5} that, for \(h\in\mathcal{H}_b\),
\[
  \Psi_A(h)
  =
  \int_{\partial\mathbb{T}_b}g_A(H_h(\xi))\,\mathbf m_b(\mathrm{d}\xi).
\]

\begin{lemma}\label{lem:Psi-Lipschitz}
The functional \(\Psi_A:\mathcal{H}_b\to[0,1]\) is globally Lipschitz: for any $h,~\widetilde h\in\mathcal{H}_b$,
\begin{equation}\label{eq:Psi-Lipschitz}
  |\Psi_A(h)-\Psi_A(\widetilde h)|
  \leq
  \sqrt{\frac{2}{\pi}}\,r_A
  \|h-\widetilde h\|_b.
\end{equation}
\end{lemma}

\begin{proof}
By \eqref{eq:gA-Lipschitz} and \eqref{eq:path-map-bound},
\[
  |\Psi_A(h)-\Psi_A(\widetilde h)|
  \leq
  \sqrt{\frac{2}{\pi}}\,r_A
  \int_{\partial\mathbb{T}_b}|H_h(\xi)-H_{\widetilde h}(\xi)|
  \,\mathbf m_b(\mathrm{d}\xi)
  \leq
  \sqrt{\frac{2}{\pi}}\,r_A\|h-\widetilde h\|_b.
\]
\end{proof}
The following definition comes from \cite[p. 4]{DemboZeitouni2010}.
\begin{definition}\label{54thytr7gt5}
A function \(I:E\to[0,\infty]\) on a metric space \(E\) is {\em lower
semicontinuous} if the sublevel set \(\{x\in E:I(x)\leq L\}\) is closed for every $L\geq0$.
Equivalently, $I$ is lower semicontinuous if and only if for every sequence $\{x_n\}_{n\geq1}\subset E$ converging to $x\in E$, it follows that $\liminf_{n\to\infty}I(x_n)\geq I(x).$ We call $I$ a {\em rate function} if it is lower semicontinuous.
Moreover, the rate function is {\em good} if \(\{x\in E:I(x)\leq L\}\) is compact for every $L\geq0$.
\end{definition}
Recall that for $h\in\mathcal{H}_b$,
\[
\mathcal{I_{\alpha}}(h):=\lambda\sum_{v\neq\varnothing}|h_v|^\alpha.
\]
The following proposition shows that $\mathcal{I_{\alpha}}$ is a good rate function, and therefore its sublevels are compact.
\begin{proposition}
\label{prop:energy-good}
 \(\mathcal{I_{\alpha}}:\mathcal{H}_b\to[0,\infty]\) is a good rate function on \((\mathcal{H}_b,\|\cdot\|_b)\).
\end{proposition}

\begin{proof}
Fix $L\geq0$. For \(M\geq1\), the finite-coordinate function
\[
  h\longmapsto
  \lambda\sum_{1\leq|v|\leq M}|h_v|^\alpha
\]
is continuous on \(\mathcal{H}_b\). Thus,
\begin{align}\label{regeythy4}
\Big\{h\in\mathcal{H}_b: \lambda\sum_{1\leq|v|\leq M}|h_v|^\alpha\leq L\Big\}
\end{align}
 is closed. Since
\begin{align}\label{regwwrth4}
\{h\in\mathcal{H}_b:\mathcal{I_{\alpha}}(h)\leq L\}=\bigcap^{\infty}_{M=1}\Big\{h\in\mathcal{H}_b: \lambda\sum_{1\leq|v|\leq M}|h_v|^\alpha\leq L\Big\},
\end{align}
we get that $\{h:\mathcal{I_{\alpha}}(h)\leq L\}$ is also closed. Thus, by Definition \ref{54thytr7gt5}, $\mathcal{I_{\alpha}}$ is lower semicontinuous.
\par
 It remains to show \(K_L:=\{h\in\mathcal{H}_b:\mathcal{I_{\alpha}}(h)\leq L\}\) is compact.  We claim that
 $$\lim_{M\to\infty} \|h-P_Mh\|_b=0$$
uniformly for \(h\in K_L\). In fact, if
\(0<\alpha\leq1\), then for \(h\in K_L\),
\begin{align}\label{they6yt5t}
  \|h-P_Mh\|_b
  &\leq
  b^{-(M+1)}
  \sum_{|v|>M}|h_v|\cr
  &\leq  b^{-(M+1)}\left(\sum_{v\neq \varnothing}|h_v|^{\alpha}\right)^{1/\alpha}
  \leq
  b^{-(M+1)}(L/\lambda)^{1/\alpha}.
\end{align}
For \(\alpha>1\), by H\"older's inequality, it follows that
\begin{align}\label{ookioq}
  \|h-P_Mh\|_b
  &\leq
  \left(\sum_{|v|>M}|h_v|^\alpha\right)^{1/\alpha}
  \left(\sum_{|v|>M}b^{-\alpha'|v|}\right)^{1/\alpha'}\cr
  &\leq
  (L/\lambda)^{1/\alpha}
  \left(\sum_{k>M}b^{-k/(\alpha-1)}\right)^{(\alpha-1)/\alpha},
\end{align}
where \(\alpha'=\alpha/(\alpha-1)\). Both bounds in \eqref{they6yt5t} and \eqref{ookioq} converge to zero as \(M\to\infty\), uniformly in \(h\in K_L\). Fix \(\varepsilon>0\). Thus, there exists \(M\geq1\) such that
\begin{align}\label{ergfbfg6y}
\sup_{h\in K_L}\|h-P_Mh\|_b<\frac{\varepsilon}{2}.
\end{align}
Recall that the finite-dimensional space \(E_M=\{h\in\mathcal{H}_b:h_v=0\text{ whenever }|v|>M\}\). Similar to \eqref{regeythy4},
$$P_MK_L=\Bigg\{g\in E_M: \lambda\sum_{1\leq|v|\leq M}|g_v|^{\alpha}\leq L\Bigg\}$$
is closed in \(E_M\). Moreover, observe that $P_MK_L\subset\{g\in E_M: \|g\|_b\leq M(\frac{L}{\lambda})^{1/\alpha}\}$. Thus, $P_MK_L$ is totally bounded and closed in \(E_M\), and hence compact.  Therefore, there exist \(g_1,\ldots,g_N\in E_M\) such that
\[
P_MK_L
\subset
\bigcup_{j=1}^{N}
B_b\left(g_j,\frac{\varepsilon}{2}\right),
\]
where
\[
B_b(g,r):=\{f\in\mathcal H_b:\|f-g\|_b<r\}.
\]
For every \(h\in K_L\), choose \(j\in\{1,\ldots,N\}\) such that
\[
\|P_Mh-g_j\|_b<\frac{\varepsilon}{2}.
\]
This, combined with \eqref{ergfbfg6y}, yields that
\[
\|h-g_j\|_b
\leq
\|h-P_Mh\|_b+\|P_Mh-g_j\|_b
<\varepsilon.
\]
Therefore,
\[
K_L\subset\bigcup_{j=1}^{N}B_b(g_j,\varepsilon),
\]
and \(K_L\) is totally bounded. Moreover, \(K_L\) is closed by \eqref{regwwrth4}, and \(\mathcal H_b\) is complete by Proposition~2.1. Thus \(K_L\) is compact.
\end{proof}
\par
Set
\begin{align}\label{rggbtht5}
  \kappa_{\alpha,b}
  :=
  \begin{cases}
    b^\alpha,&0<\alpha\leq1,\\[0.3em]
    \bigl(b^{1/(\alpha-1)}-1\bigr)^{\alpha-1},&\alpha>1.
  \end{cases}
\end{align}
The following proposition gives explicit bounds for $Q_A$ and justifies its attainment. Recall that
\begin{align}\label{5tgtwg5gyh}
  Q_A(p)
  =
  \inf\Big\{\lambda\sum_{v\neq\varnothing}|h_v|^\alpha:h\in\mathcal{H}_b,\ \Psi_A(h)=p\Big\},
  \qquad p\in[0,1].
\end{align}
\begin{proposition}
\label{prop:cost-bounds}
(i) For every \(p\in(\nu(A),M_A)\),
\begin{equation}\label{eq:explicit-Q-bounds}
  \lambda\kappa_{\alpha,b}
  \left(
    \frac{p-\nu(A)}
    {\sqrt{2/\pi}\,r_A}
  \right)^\alpha
  \leq
  Q_A(p)
  \leq
  \lambda b\,I_A(p)^\alpha
  <\infty.
\end{equation}
(ii) $Q_A$ is lower semicontinuous on $[0,1]$.\\
(iii) For $q\in[0,1]$, if \(Q_A(q)<\infty\), then the infimum defining \(Q_A(q)\) is
attained. Moreover, for \(p\in(\nu(A),M_A)\), the closed-constraint problem
\[
  \inf\{\mathcal{I_{\alpha}}(h):\Psi_A(h)\geq p\}
\]
is also attained.
\end{proposition}

\begin{proof}We first prove the lower bound in \eqref{eq:explicit-Q-bounds}.
If $h\in\mathcal{H}_b$ and \(\Psi_A(h)=p\), Lemma~\ref{lem:Psi-Lipschitz} gives
\begin{align}\label{5yhyh6t5}
  |p-\nu(A)|=|\Psi_A(h)-\Psi_A(\mathbf{0})|
  \leq
  \sqrt{\frac{2}{\pi}}\,r_A\|h\|_b,
\end{align}
where $\mathbf{0}$ stands for the zero field (i.e. $\mathbf{0}_v=0$ for any $v\in\mathbb{T}_b$).
If \(0<\alpha\leq1\), then
\[
  \|h\|_b
  \leq
  b^{-1}\sum_{v\neq\varnothing}|h_v|
  \leq
  b^{-1}
  \left(\sum_{v\neq\varnothing}|h_v|^\alpha\right)^{1/\alpha}.
\]
This, together with \eqref{5yhyh6t5} and \eqref{rggbtht5}, entails that if \(0<\alpha\leq1\), then for $h\in\mathcal{H}_b$ and \(\Psi_A(h)=p\),
\begin{align}\label{4gtre55t}
\lambda\left(\sum_{v\neq\varnothing}|h_v|^\alpha\right)\geq \lambda b^{\alpha}(\|h\|_b)^{\alpha}\geq \lambda\kappa_{\alpha,b}
  \left(
    \frac{p-\nu(A)}
    {\sqrt{2/\pi}\,r_A}
  \right)^\alpha.
\end{align}
Then, by the definition of $Q_A$ defined in \eqref{5tgtwg5gyh}, we get that for \(0<\alpha\leq1\),
\begin{align}\label{retgbtr5gth}
Q_A(p)\geq\lambda\kappa_{\alpha,b}
  \left(
    \frac{p-\nu(A)}
    {\sqrt{2/\pi}\,r_A}
  \right)^\alpha.
\end{align}
If \(\alpha>1\), H\"older's inequality gives
\[
  \|h\|_b
  \leq
  \left(\sum_{v\neq\varnothing}|h_v|^\alpha\right)^{1/\alpha}
  \left(\sum_{k\geq1}b^{-k/(\alpha-1)}\right)^{(\alpha-1)/\alpha}.
\]
Then, using similar arguments to \eqref{4gtre55t}-\eqref{retgbtr5gth}, we obtain the first inequality in
\eqref{eq:explicit-Q-bounds} for \(\alpha>1\).
\par
We proceed to deal with the upper bound in (i). For every \(p\in(\nu(A),M_A)\), by \eqref{4gr5yhy6yr4}, the continuity of $g_A$ and \eqref{4he6hyh62}, we can choose \(x\) such that \(g_A(x)=p\) and $|x|=I_A(p)$. Set $h_v=x$ for $|v|=1$ and $h_v=0$ for $|v|>1$. In this setting, we have $\Psi_A(h)=p$ and $\mathcal{I_{\alpha}}(h)=\lambda b|x|^\alpha$.  This entails that
\begin{align}\label{tryj66tyj}
  Q_A(p)\leq\lambda b I_A(p)^\alpha.
\end{align}
The latter quantity is finite because \(p<M_A\).
\par

Next, we show that $Q_A$ is lower semicontinuous. By Definition \ref{54thytr7gt5}, it suffices to show that for any $L\geq0$,
\[
 \{q\in[0,1]:Q_A(q)\leq L\}
\]
is closed. Let $q_n\in\{q\in[0,1]:Q_A(q)\leq L\},~n\geq1$ and $\lim_{n\to\infty}q_n=q^*$.
By the definition of \(Q_A(q_n)\), there exists
\(h_n\in\mathcal{H}_b\) such that
\[
  \Psi_A(h_n)=q_n,
  \qquad
  \mathcal{I_{\alpha}}(h_n)\leq Q_A(q_n)+\frac1n
  \leq L+\frac1n.
\]
Hence, for any $n\geq1$,
\[
  h_n\in\{h\in\mathcal{H}_b:\mathcal{I_{\alpha}}(h)\leq L+1\}.
\]
This set is compact by Proposition~\ref{prop:energy-good}. Thus, there exists $h^*\text{in }\mathcal{H}_b$ such that along a subsequence $\{n_j\}_{j\geq1}$,
\begin{align}\label{4tgty6gr4tyr4}
  \lim_{j\to\infty}h_{n_j}=h^*.
\end{align}
The continuity of \(\Psi_A\) (see Lemma \ref{lem:Psi-Lipschitz}) and the lower semicontinuity of \(\mathcal{I_{\alpha}}\)
give
\[
  \Psi_A(h^*)
  =\lim_{j\to\infty}\Psi_A(h_{n_j})
  =\lim_{j\to\infty}q_{n_j}
  =q^*
\]
and
\[
  \mathcal{I_{\alpha}}(h^*)
  \leq\liminf_{j\to\infty}\mathcal{I_{\alpha}}(h_{n_j})
  \leq L.
\]
Consequently,
\[
  Q_A(q^*)\leq\mathcal{I_{\alpha}}(h^*)\leq L.
\]
This shows that $q^*\in\{q\in[0,1]:Q_A(q)\leq L\}$, which means $\{q\in[0,1]:Q_A(q)\leq L\}$ is closed.
\par
We now justify attainment in (iii). If \(Q_A(q)<\infty\), then there exists \((h_n)_{n\geq1}\subset\{h\in\mathcal{H}_b:\ \Psi_A(h)=q\}\) such that $\lim_{n\to\infty}\mathcal{I_{\alpha}}(h_n)=Q_A(q)$ . Thus, there exists $N$ such that for $n>N$, $(h_n)_{n\geq N}\subset\{h\in\mathcal{H}_b: \mathcal{I_{\alpha}}(h)\leq Q_A(q)+1\}$.
By Proposition~\ref{prop:energy-good}, $\{h\in\mathcal{H}_b: \mathcal{I_{\alpha}}(h)\leq Q_A(q)+1\}$ is compact. Hence, there exists a subsequence \((h_{n_k})_{k\geq1}\) that converges in \(\mathcal{H}_b\) to some \(h\).  Continuity of \(\Psi_A\)
gives \(\Psi_A(h)=q\), and lower semicontinuity of \(\mathcal{I_{\alpha}}\) gives
\[
  \mathcal{I_{\alpha}}(h)\leq\liminf_{k\to\infty}\mathcal{I_{\alpha}}(h_{n_k})=Q_A(q).
\]
Thus the infimum defining $Q_A(q)$ is attained at $h$. The attainment of $\inf\{\mathcal{I_{\alpha}}(h):\Psi_A(h)\geq p\}$ follows from a similar argument.
\end{proof}

\subsection{Standard large-deviation tools}\label{4gtgetg4r}

We collect several standard large-deviation tools from Dembo and
Zeitouni~\cite{DemboZeitouni2010}. They are recorded here to fix the precise
forms needed below. Let \((E,d)\) be a metric space. The following definition comes from \cite[p. 5, p. 7]{DemboZeitouni2010}.
\begin{definition}[Large-deviation principles]
\label{def:LDP}
Let \(\{V_n\}_{n\geq1}\) be \(E\)-valued Borel
random elements and $\{a_n\}_{n\geq1}$ be a sequence tending to infinity. A sequence
\(\{V_n\}_{n\geq1}\) satisfies a \emph{(full) large-deviation principle}
on \(E\), with speed \(a_n\) and rate function \(I\), if, for every
Borel set \(\Gamma\subset E\),
\begin{align}
  -\inf_{x\in\Gamma^\circ}I(x)
  &\leq
  \liminf_{n\to\infty}a_n^{-1}
  \log\mathbb{P}(V_n\in\Gamma)
  \notag\\
  &\leq
  \limsup_{n\to\infty}a_n^{-1}
  \log\mathbb{P}(V_n\in\Gamma)
  \leq
  -\inf_{x\in\overline\Gamma}I(x).\nonumber
\end{align}
It satisfies a \emph{weak large-deviation principle} if
\begin{align}
  \liminf_{n\to\infty}a_n^{-1}\log\mathbb{P}(V_n\in G)
  &\geq-\inf_{x\in G}I(x)
  \label{eq:weak-LDP-lower}
\end{align}
for every open set \(G\subset E\), and
\begin{align}
  \limsup_{n\to\infty}a_n^{-1}\log\mathbb{P}(V_n\in K)
  &\leq-\inf_{x\in K}I(x)
  \label{eq:weak-LDP-upper}
\end{align}
for every compact set \(K\subset E\).
\end{definition}
The following definition comes from \cite[p. 8]{DemboZeitouni2010}.
\begin{definition}
The sequence \(\{V_n\}_{n\geq1}\) is \emph{exponentially tight} at
speed \(a_n\) if, for every \(L>0\), there exists a compact set
\(K_L\subset E\) such that
\begin{equation}\label{eq:exponential-tightness}
  \limsup_{n\to\infty}a_n^{-1}
  \log\mathbb{P}(V_n\notin K_L)
  \leq-L.
\end{equation}
\end{definition}

In the following, we recall the standard upgrade from a weak LDP to a full LDP;
see \cite[Lemma~1.2.18]{DemboZeitouni2010}.

\begin{lemma}[Weak LDP and exponential tightness]
\label{lem:weak-to-full}
If a sequence satisfies a weak LDP at speed \(a_n\) and is exponentially
tight at that speed, then it satisfies a full LDP and its rate function is
good.
\end{lemma}

The following open-ball criterion is copied from
\cite[Theorem~4.1.11]{DemboZeitouni2010}. Let $B(x,\delta):=\{y\in E: d(y,x)<\delta\}$ be the open-ball centered at $x\in E$ with radius $\delta>0$.

\begin{lemma}[Local LDP criterion]\label{lem:local-LDP}
Let \(\{V_n\}_{n\geq1}\) be random variables in a metric space \(E\). If for every
\(x\in E\),
\begin{align*}
  -I(x)
  &=
  \lim_{\delta\to0+}\liminf_{n\to\infty}
  a_n^{-1}\log\mathbb{P}(V_n\in B(x,\delta))\\
  &=
  \lim_{\delta\to0+}\limsup_{n\to\infty}
  a_n^{-1}\log\mathbb{P}(V_n\in B(x,\delta)),
\end{align*}
then \(\{V_n\}_{n\geq1}\) satisfies a weak LDP with rate function \(I\).
\end{lemma}
The following contraction principle is borrowed from \cite[Theorem~4.2.1]{DemboZeitouni2010}. Let $(E',d')$ be some metric space.
\begin{lemma}[Contraction principle]\label{lem:contraction}
If \(\{V_n\}_{n\geq1}\) satisfies a full LDP on \(E\), with speed \(a_n\) and good rate function
\(I\), and \(F:E\to E'\) is continuous, then \(\{F(V_n)\}_{n\geq1}\) satisfies a full
LDP with good rate function
\[
  I_F(y):=\inf\{I(x):F(x)=y,~x\in E \}.
\]
\end{lemma}
The following definition can be found in \cite[Definition 4.2.10]{DemboZeitouni2010}.
\begin{definition}\label{4grtet6y6yf}
Two sequences \(\{V_n\}_{n\geq1}\) and \(\{W_n\}_{n\geq1}\) in the same metric space are exponentially
equivalent at speed \(a_n\) if, for every \(\delta>0\),
\[
  \limsup_{n\to\infty}a_n^{-1}
  \log\mathbb{P}(d(V_n,W_n)>\delta)=-\infty.
\]
\end{definition}

The transfer results for exponentially equivalent sequences can be found in
\cite[Theorem~4.2.13]{DemboZeitouni2010}.

\begin{lemma}[Exponential equivalence]\label{lem:exp-equivalence-transfer}
Assume \(\{V_n\}_{n\geq1}\) and \(\{W_n\}_{n\geq1}\) are exponentially
equivalent at speed \(a_n\), and \(\{V_n\}_{n\geq1}\) satisfies a full LDP with speed \(a_n\) and good rate function
\(I\). Then, \(\{W_n\}_{n\geq1}\) also satisfies a full LDP with speed \(a_n\) and good rate function
\(I\).
\end{lemma}

The approximation from finite-depth fields to infinite fields will use the following
exponentially-good-approximation lemma
\cite[Theorem~4.2.16]{DemboZeitouni2010}.

\begin{lemma}[Exponentially good approximations]
\label{lem:good-approximation-DZ}
Suppose that, for every fixed \(M\), \(\{V_n^{(M)}\}_{n\geq1}\) satisfies an LDP with
rate function \(I_M\), and
\[
  \lim_{M\to\infty}\limsup_{n\to\infty}a_n^{-1}
  \log\mathbb{P}(d(V_n^{(M)},V_n)>\delta)=-\infty
\]
for every \(\delta>0\). Then \(\{V_n\}_{n\geq1}\) satisfies a weak LDP with candidate
rate
\[
  \widetilde I(x)
  =
  \sup_{\delta>0}\liminf_{M\to\infty}
  \inf_{y\in B(x,\delta)}I_M(y).
\]
For a closed set \(F\subset E\), put
\(F^\delta:=\{y\in E:d(y,F)<\delta\}\). If \(\widetilde I\) is good and
for every closed \(F\subset E\),
\[
  \lim_{\delta\to0+}\liminf_{M\to\infty}
  \inf_{y\in F^\delta}I_M(y)
  \geq \inf_{x\in F}\widetilde I(x),
\]
then the LDP is full.
\end{lemma}

\section{Proofs of the main theorems and properties of the rate function}
\label{sec:proofs}

The main ideas of the proof are as follows. To prove Theorem \ref{thm:empirical-LDP}, we first establish an LDP for the truncated
rescaled displacement fields $\{P_MY_n\}_{n\geq1}$ in Lemma \ref{54gtrehr6r}, and show that $\{P_MY_n\}_{n\geq1}$ is an
exponentially good approximation for $\{Y_n\}_{n\geq1}$; see Proposition \ref{prop:exp-approx}. Hence, we can obtain a full LDP for $\{Y_n\}_{n\geq1}$; see Theorem \ref{thm:tree-field-LDP}. After that, we found $\{Z_n(\sqrt n\,A)\}_{n\geq1}$ can be well approximated by the Gaussian functions $\{R_n
\}_{n\geq1}$; see Lemma \ref{lem:conditional-concentration}. Furthermore, in Lemma \ref{lem:kernel-replacement}, we show that \(\{R_n\}_{n\geq1}\) and \(\{\Psi_A(Y_n)\}_{n\geq1}\) are exponentially equivalent by the continuity of the Gaussian measure. Hence, $\{Z_n(\sqrt n\,A)\}_{n\geq1}$ and \(\{\Psi_A(Y_n)\}_{n\geq1}\) satisfy the same LDP, which can be obtained by the contraction principle and the LDP of $\{Y_n\}_{n\geq1}$ proved in Theorem \ref{thm:tree-field-LDP}. This completes the proof of Theorem \ref{thm:empirical-LDP}. Theorem~\ref{thm:main-one-sided} is a direct application of Theorem \ref{thm:empirical-LDP}. Finally, we give several properties of the resulting variational rate function $Q_A$.

\subsection{The LDP for rescaled displacement fields}

For $n\geq1$, let
\begin{equation}\label{eq:kn-mn}
  k_n:=\min\{\lfloor c\log n\rfloor, n-1\},
  \qquad
  m_n:=n-k_n>0,
\end{equation}
where
\begin{equation}\label{eq:c-choice}
  \begin{cases}
    \displaystyle
    \frac{\alpha}{2\log b}
    <c<
    \frac{\alpha}{2(1-\alpha)\log b},
      &0<\alpha<1,\\[0.9em]
    \displaystyle
    c>\frac{\alpha}{2\log b},
      &\alpha\geq1.
  \end{cases}
\end{equation}
The interval $[\frac{\alpha}{2\log b},\frac{\alpha}{2(1-\alpha)\log b}]$ is nonempty.  In all cases,
\begin{equation}\label{eq:kn-properties}
  \lim_{n\to\infty}\frac{b^{k_n}}{n^{\alpha/2}}=\infty,
\end{equation}
and, when \(0<\alpha<1\),
\begin{equation}\label{eq:kn-sublinear-property}
  b^{(1-\alpha)k_n}=o(n^{\alpha/2}).
\end{equation}
For $n\geq1$, define the {\em rescaled displacement field} $Y_n$ as follows:
\begin{equation}\label{eq:Yn}
  Y_n(v):=
  \frac{X_v}{\sqrt{m_n}}\,
  \mathbf{1}_{\{|v|\leq k_n\}},
  \qquad v\in\mathbb{T}_b^\circ,
\end{equation}
which is \(\mathcal{H}_b\)-valued.
\par
We first give an LDP for a sequence of rescaled Weibull random variables.
\begin{lemma}\label{lem:one-dimensional-LDP}
The variables \(X/\sqrt{m_n}, n\geq1\) satisfy a full LDP on \(\mathbb{R}\), with speed
\(n^{\alpha/2}\) and good rate function
\[
  I_\alpha(x):=\lambda|x|^\alpha,~x\in\mathbb{ R}.
\]
\end{lemma}

\begin{proof}
By Assumption \ref{45twzg5yy43} (2) and (3), for any \(t>0\), it follows that
\[
  \lim_{n\to\infty}m_n^{-\alpha/2}
  \log\mathbb{P}\!\Big(\frac{X}{\sqrt{m_n}}>t\Big)
  =
  \lim_{n\to\infty}m_n^{-\alpha/2}
  \log\mathbb{P}\!\Big(\frac{X}{\sqrt{m_n}}<-t\Big)
  =-\lambda t^\alpha .
\]
For \(x\neq0\) and \(0<\delta<|x|\), above yields that
\[
  \lim_{n\to\infty}m_n^{-\alpha/2}
  \log\mathbb{P}\!\Big(
    \frac{X}{\sqrt{m_n}}\in(x-\delta,x+\delta)
  \Big)
  =-\lambda(|x|-\delta)^\alpha .
\]
For $x=0$, since $X$ is a finite random variable and $\lim_{n\to\infty}m_n=\infty$, we have
\[
  \lim_{n\to\infty}m_n^{-\alpha/2}
  \log\mathbb{P}\!\Big(
    \frac{X}{\sqrt{m_n}}\in(-\delta,\delta)
  \Big)
  =0 .
\]
With this in hand, it is easy to check that \(X/\sqrt{m_n}, n\geq1\) satisfies the condition of Lemma~\ref{lem:local-LDP}. Hence, we obtain a weak LDP for \(X/\sqrt{m_n}, n\geq1\) with speed $m_n^{\alpha/2}$ and rate function $I_{\alpha}$. Finally, for any $L>0$, we have
\[
  \limsup_{n\to\infty}m_n^{-\alpha/2}
  \log\mathbb{P}\!\Big(\frac{|X|}{\sqrt{m_n}}>L\Big)
  =-\lambda L^\alpha ,
\]
which means the sequence is exponentially tight; see Definition \ref{eq:exponential-tightness}. Since $\lim_{n\to\infty}\frac{m_n}{n}=1$, applying Lemma~\ref{lem:weak-to-full}, the lemma follows.
\end{proof}

Recall that for $n\geq1$ and $v\in\mathbb{T}_b^\circ$,
\begin{align}\label{45t46yy5t}
k_n:=\lfloor c\log n\rfloor,\quad
  Y_n(v):=
  \frac{X_v}{\sqrt{m_n}}\,
  \mathbf{1}_{\{|v|\leq k_n\}},\quad
  P_Mh_v:=h_v\mathbf{1}_{\{|v|\leq M\}}.
\end{align}
The following proposition shows that $\{(P_MY_n)_{n\geq1}\}_{M\geq1}$ is an exponentially good approximation for $\{Y_n\}_{n\geq1}$.
\begin{proposition}
\label{prop:exp-approx}
For every \(\delta>0\),
\begin{equation}\label{eq:exp-approx}
  \lim_{M\to\infty}\limsup_{n\to\infty}
  n^{-\alpha/2}
  \log\mathbb{P}\!\left(\|Y_n-P_MY_n\|_b>\delta\right)
  =-\infty.
\end{equation}
\end{proposition}

\begin{proof}
Note that
\[
  \|Y_n-P_MY_n\|_b
  =
  \frac1{\sqrt{m_n}}
  \sum_{k=M+1}^{k_n}b^{-k}\sum_{|v|=k}|X_v|.
\]
We first deal with the case \(0<\alpha\leq1\). Since
\((x+y)^\alpha\leq x^\alpha+y^\alpha\) for \(x,y\geq0\),
\begin{align}\label{54gyhu7t5}
  \|Y_n-P_MY_n\|_b^\alpha
  \leq
  m_n^{-\alpha/2}
  \sum_{k=M+1}^{k_n}b^{-\alpha k}
  \sum_{|v|=k}|X_v|^\alpha .
\end{align}
By \eqref{eq:tail}, $\mathbb{E} e^{\theta|X|^\alpha}<\infty$ for \(\theta<\lambda\). Since \(F(\theta)=\log\mathbb{E} e^{\theta|X|^\alpha}\) is convex on $(-\infty,\lambda)$ and \(F(0)=0\), we obtain that for any $\theta_0<\lambda$,
\[
  F(\theta)=F\left(\frac{1-\theta}{\theta_0}0+\frac{\theta}{\theta_0}\theta_0\right)\leq\frac{\theta}{\theta_0}F(\theta_0),
  \qquad 0\leq\theta\leq\theta_0.
\]
This means for every \(\theta_0<\lambda\), there exists a finite
\(\rho=\rho_{\theta_0}=\frac{F(\theta_0)}{\theta_0}\) such that
\begin{align}\label{5gbgy6r4}
  \log\mathbb{E} [e^{\theta|X|^\alpha}]\leq  \rho\theta,
  \qquad 0\leq\theta\leq\theta_0.
\end{align}
Let
$\theta_0=\lambda/4$ in \eqref{5gbgy6r4}. By exponential Markov's inequality and \eqref{54gyhu7t5}, we obtain that
\begin{align}\label{4grtg6yr4}
  &\mathbb{P}(\|Y_n-P_MY_n\|_b>\delta)\cr
  &\leq
  \mathbb{E}\left[\exp\Big\{
(\lambda/4)b^{\alpha(M+1)}m_n^{\alpha/2}\Big(m_n^{-\alpha/2}
  \sum_{k=M+1}^{k_n}b^{-\alpha k}
  \sum_{|v|=k}|X_v|^\alpha-\delta^{\alpha}\big)\Big\}\right]\cr
&\leq\exp\left\{
    -(\lambda/4)b^{\alpha(M+1)}\delta^\alpha m_n^{\alpha/2}
    +\rho(\lambda/4)b^{\alpha(M+1)}\sum_{k=M+1}^{k_n}b^{(1-\alpha)k}
  \right\},
\end{align}
where the last inequality follows from \eqref{5gbgy6r4} and independence of $X_v$, $v\in\mathbb{T}_b^\circ$. For \(\alpha=1\), \(\sum_{k=M+1}^{k_n}b^{(1-\alpha)k}=O(\log n)\). For \(0<\alpha<1\),
\(\sum_{k=M+1}^{k_n}b^{(1-\alpha)k}=o(n^{\alpha/2})\) by \eqref{eq:kn-sublinear-property}.  Hence, for any
\(M\geq1\),
\[
  \limsup_{n\to\infty}n^{-\alpha/2}
  \log\mathbb{P}(\|Y_n-P_MY_n\|_b>\delta)
  \leq-(\lambda/4)b^{\alpha(M+1)}\delta^\alpha .
\]
Thus, \eqref{eq:exp-approx} follows by letting $M\to\infty.$
\par
We proceed to deal with \(\alpha>1\). Let
$
  \ell_{*}:=\sum_{\ell\geq1}\ell^{-2}.
$
For \(1\leq j\leq k_n-M\), set
\[
  a_j:=b^{-(M+j)}
  \sum_{|v|=M+j}|X_v|,
  \qquad
  w_j:=
  \frac{j^{-2}}{\sum_{\ell=1}^{k_n-M}\ell^{-2}}.
\]
By Jensen's inequality,
\begin{align*}
  \Big(\sum_{j=1}^{k_n-M}a_j\Big)^\alpha
  =
  \Big(
    \sum_{j=1}^{k_n-M}w_j\frac{a_j}{w_j}
  \Big)^\alpha
  \leq
  \sum_{j=1}^{k_n-M}w_j
  \Big(\frac{a_j}{w_j}\Big)^\alpha
  \leq
  \ell_{*}^{\alpha-1}
  \sum_{j=1}^{k_n-M}j^{2(\alpha-1)}a_j^\alpha.
\end{align*}
A second application of Jensen's inequality gives
\[
  \Big(
    \sum_{|v|=M+j}|X_v|
  \Big)^\alpha={b^{\alpha(M+j)}}\Big(
    \sum_{|v|=M+j}\frac{|X_v|}{b^{(M+j)}}
  \Big)^\alpha
  \leq
  b^{(M+j)(\alpha-1)}
  \sum_{|v|=M+j}|X_v|^\alpha.
\]
Consequently,
\begin{align}\label{eq:alpha-large-Jensen}
  \|Y_n-P_MY_n\|_b^\alpha&=
  \frac1{m_n^{\alpha/2}}
  \Big(\sum_{j=1}^{k_n-M}b^{-(M+j)}\sum_{|v|=M+j}|X_v|\Big)^{\alpha}\cr
  &\leq
  \frac{\ell_{*}^{\alpha-1}}{m_n^{\alpha/2}}
  \sum_{j=1}^{k_n-M}j^{2(\alpha-1)}\Big(b^{-(M+j)}\sum_{|v|=M+j}|X_v|\Big)^{\alpha}\cr
&\leq
  \frac{\ell_{*}^{\alpha-1}}{m_n^{\alpha/2}}
  \sum_{j=1}^{k_n-M}
  j^{2(\alpha-1)}b^{-(M+j)}
  \sum_{|v|=M+j}|X_v|^\alpha.
\end{align}
Set
\[
  D_{\alpha,b}
  :=
  \ell_{*}^{\alpha-1}
  \sup_{j\geq1}
  j^{2(\alpha-1)}b^{-j}
  <\infty,
  \qquad
  \theta_M:=
  \frac{\lambda b^M}{4D_{\alpha,b}}.
\]
Again, set $\theta_0=\lambda/4$ in \eqref{5gbgy6r4}. For every \(j\geq1\),
\[
  \theta_M \ell_{*}^{\alpha-1}
  j^{2(\alpha-1)}b^{-(M+j)}
  \leq\frac{\lambda}{4}.
\]
Similar to \eqref{4grtg6yr4}, by \eqref{eq:alpha-large-Jensen}, it follows that
\begin{align*}
  &\mathbb{P}\!\left(\|Y_n-P_MY_n\|_b>\delta\right)\\
  &\leq
  \mathbb{E}\left[\exp\Big\{
\theta_Mm_n^{\alpha/2}\Big(\frac{\ell_{*}^{\alpha-1}}{m_n^{\alpha/2}}
  \sum_{j=1}^{k_n-M}
  j^{2(\alpha-1)}b^{-(M+j)}
  \sum_{|v|=M+j}|X_v|^\alpha-\delta^{\alpha}\Big)\Big\}\right]\cr
  &\leq
  \exp\left\{
    -\theta_M\delta^\alpha m_n^{\alpha/2}
    +\rho\theta_M \ell_{*}^{\alpha-1}
      \sum_{j=1}^{k_n-M}j^{2(\alpha-1)}
  \right\}.
\end{align*}
Since \(k_n=O(\log n)\), we have
$
  \sum_{j=1}^{k_n-M}j^{2(\alpha-1)}
  =
  O\bigl((\log n)^{2\alpha-1}\bigr).
$
Hence, above yields that
\[
  \limsup_{n\to\infty}
  n^{-\alpha/2}
  \log\mathbb{P}\!\left(\|Y_n-P_MY_n\|_b>\delta\right)
  \leq-\theta_M\delta^\alpha.
\]
The desired result follows by letting $M\to\infty$.
\end{proof}
For fixed \(M\geq1\), define
\begin{equation}\label{eq:finite-energy}
  \mathcal{I_{\alpha}}^{(M)}(h):=
  \begin{cases}
    \displaystyle
    \lambda\sum_{1\leq|v|\leq M}|h_v|^\alpha,
       &h\in E_M,\\
    \infty,&h\notin E_M,
  \end{cases}
\end{equation}
where we recall that $E_M=\{h\in\mathcal{H}_b:h_v=0\text{ whenever }|v|>M\}.$ Recall from \eqref{eq:Yn} that \(Y_n\) is supported on the first \(k_n\)
generations, while \(P_MY_n\) retains only its first \(M\) generations.

By Lemma \ref{lem:one-dimensional-LDP}, finite product principle for large deviations (see \cite[Exercise 4.2.7]{DemboZeitouni2010}) and independence, we obtain the following lemma.

\begin{lemma}\label{54gtrehr6r}
\(\{P_MY_n\}_{n\geq1}\), viewed as a sequence of random variables in \(\mathcal{H}_b\), satisfies a full LDP with speed
\(n^{\alpha/2}\) and rate \(\mathcal{I_{\alpha}}^{(M)}\).
\end{lemma}

In the following theorem, by letting $M\to\infty$ to \(\{P_MY_n\}_{n\geq1}\), we obtain a full LDP for \(\{Y_n\}_{n\geq1}\).
\begin{theorem}\label{thm:tree-field-LDP}
The fields \(\{Y_n\}_{n\geq1}\) satisfy a full LDP on \(\mathcal{H}_b\), with speed
\(n^{\alpha/2}\) and good rate function \(\mathcal{I_{\alpha}}\).
\end{theorem}
\begin{proof}
By Proposition~\ref{prop:exp-approx}, \(\{P_MY_n\}_{n\geq1}\) is an
exponentially good approximation for \(\{Y_n\}_{n\geq1}\). Therefore, applying Lemma~\ref{lem:good-approximation-DZ} and Lemma \ref{54gtrehr6r}, we get a weak LDP for $\{Y_n\}_{n\geq1}$ with candidate rate
\[
  \widetilde{\mathcal{I_{\alpha}}}(h)
  =
  \sup_{\delta>0}\liminf_{M\to\infty}
  \inf_{g\in B_b(h,\delta)}\mathcal{I_{\alpha}}^{(M)}(g),
\]
where \(B_b(h,\delta)\) is the open \(\|\cdot\|_b\)-ball.
\par
We claim that $\widetilde{\mathcal{I_{\alpha}}}(h)=\mathcal{I_{\alpha}}(h).$ For every \(h\in\mathcal{H}_b\), since \(\lim_{M\to\infty}P_Mh=h\) (see Proposition \ref{prop:Hb-Banach}), it follows that
\[
  \widetilde{\mathcal{I_{\alpha}}}(h)
  \leq
  \lim_{M\to\infty}\mathcal{I_{\alpha}}^{(M)}(P_Mh)
  =\mathcal{I_{\alpha}}(h).
\]
If $\mathcal{I_{\alpha}}(h)=0$, then above yields that $\widetilde{\mathcal{I_{\alpha}}}(h)=\mathcal{I_{\alpha}}(h)=0.$ Hence, it suffices to consider $\mathcal{I_{\alpha}}(h)>0$. Fix any \(L\in(0,\mathcal{I_{\alpha}}(h))\). The compact set
\(K_L=\{g:\mathcal{I_{\alpha}}(g)\leq L\}\) does not contain \(h\), so
\(\operatorname{dist}_{\|\cdot\|_b}(h,K_L)>0\).  Choose \(\delta^*<\operatorname{dist}_{\|\cdot\|_b}(h,K_L)\), then every \(g\in B_b(h,\delta^*)\) satisfies
\(\mathcal{I_{\alpha}}(g)>L\).  Since \(\mathcal{I_{\alpha}}^{(M)}(g)\geq\mathcal{I_{\alpha}}(g)\) for all \(g\), we get
\[
  \widetilde{\mathcal{I_{\alpha}}}(h)
  \geq\liminf_{M\to\infty}
  \inf_{g\in B_b(h,\delta^*)}\mathcal{I_{\alpha}}^{(M)}(g) \geq\liminf_{M\to\infty}
  \inf_{g\in B_b(h,\delta^*)}\mathcal{I_{\alpha}}(g)\geq L.
\]
Since above holds for any \(L\in(0,\mathcal{I_{\alpha}}(h))\), it follows that $\widetilde{\mathcal{I_{\alpha}}}(h)\geq\mathcal{I_{\alpha}}(h)$. Hence, $\widetilde{\mathcal{I_{\alpha}}}(h)=\mathcal{I_{\alpha}}(h),$ which is a good rate function by Proposition \ref{prop:energy-good}.
\par
It remains to show the LDP for $\{Y_n\}_{n\geq1}$ is full. Since \(\mathcal{I_{\alpha}}^{(M)}\geq\mathcal{I_{\alpha}}\), for every closed \(F\subset\mathcal{H}_b\) and \(\delta>0\), it holds that
\[
  \lim_{\delta\to0+}\liminf_{M\to\infty}\inf_{g\in F^\delta}\mathcal{I_{\alpha}}^{(M)}(g)
  \geq \lim_{\delta\to0+}\inf_{g\in F^\delta}\mathcal{I_{\alpha}}(g)=\inf_{g\in F}\mathcal{I_{\alpha}}(g),
\]
where the last equality follows from the goodness of \(\mathcal{I_{\alpha}}\).
Hence the closed-set condition in
Lemma~\ref{lem:good-approximation-DZ} is satisfied. Finally, by Lemma~\ref{lem:good-approximation-DZ}, the LDP is full.
\end{proof}

\subsection{Proof of Theorem~\ref{thm:empirical-LDP}}

Let \(\nu_n\) be the law of \(X_1+\cdots+X_n\), $n\geq1$. The following lemma gives
a uniform central-limit estimate. This result is a combination of
\cite[Lemma~2.2]{LouidorPerkins2015} and \cite[Lemma 2.1 (i)]{ChenHe2019}.

\begin{lemma}
\label{lem:uniform-CLT}
For every \(\ell>1\),
\begin{equation}
  \lim_{n\to\infty}
  \sup_{a\in[\ell^{-1},\ell]}\sup_{z\in\mathbb{R}}
  \left|
    \nu_n\!\left(\sqrt n\,(aA-z)\right)-\nu(aA-z)
  \right|
  =0.\nonumber
\end{equation}
Moreover, \((a,z)\mapsto\nu(aA-z)\) is uniformly continuous on $K\times \mathbb{R}$ for every compact set $K\subset(0,\infty)$.
\end{lemma}

We next state the uniform concentration
estimate in \cite[Lemma~2.4]{LouidorPerkins2015}. For a nonzero finite point measure \(\zeta\) on \(\mathbb{R}\), denote by $\{Z_n^{\zeta}\}_{n\geq0}$ the branching random walk with initial value $\zeta.$ Put
\[
  \overline Z_n^\zeta(\cdot)
  :=\frac{Z_n^\zeta(\cdot)}{Z^{\zeta}_n(\mathbb{R})},
  \qquad |\zeta|:=\zeta(\mathbb{R}).
\]
We write $x\in\zeta$ if $x$ is an atom of $\zeta.$
\begin{lemma}
\label{lem:finite-configuration-concentration}
There exists $C_1,C_2>0$ such that for all $\delta>0$ sufficiently small, measurable set \(B\subset\mathbb{R}\), \(n\geq1\) and
nonzero finite point measure \(\zeta\), it follows that
\[
  \mathbb{P}\!\left(
    \Bigg|
      \overline Z_n^\zeta(B)
      -\frac1{|\zeta|}\sum_{x\in\zeta}\nu_n(B-x)
    \Bigg|>\delta
  \right)
  \leq C_1\exp\{-C_2\delta^2|\zeta|\}.
\]
\end{lemma}
Let \(\mathcal{F}_{n}:=\sigma(Z_k,0\leq k\leq n),~n\geq0\) be the natural filtration of the branching random walk. Define
\begin{equation}
  R_n
  :=
  \frac{1}{|Z_{k_n}|}\sum_{u\in Z_{k_n}}
  \nu_{n-k_n}\!\left(\sqrt n\,A-S_u\right),
\end{equation}
where $k_n$ is defined in \eqref{eq:kn-mn}. The following lemma says that $\{\overline Z_n(\sqrt n\,A)\}_{n\geq1}$ and $\{R_n\}_{n\geq1}$ are exponentially equivalent; see Definition \ref{4grtet6y6yf}.
\begin{lemma}\label{lem:conditional-concentration}
For every \(\delta>0\),
\[
  \limsup_{n\to\infty}n^{-\alpha/2}
  \log\mathbb{P}\!\left(
    \left|\overline Z_n(\sqrt n\,A)-R_n\right|>\delta
  \right)
  =-\infty.
\]
\end{lemma}

\begin{proof}Let $\{\mathcal{Z}_n\}_{n\geq0}$ be an independent copy of the branching random walk $\{{Z}_n\}_{n\geq0}$. By Lemma \ref{lem:finite-configuration-concentration}, for $\delta>0$ sufficiently
small, it follows that for any $n\geq1$,
\begin{align}
\mathbb{P}\!\left(
    \left|\overline Z_n(\sqrt n\,A)-R_n\right|>\delta
  \right)
  &=\mathbb{E}\!\left[\mathbb{P}\!\left(\left|\overline Z_n(\sqrt n\,A)-R_n\right|>\delta\,\middle|\,\mathcal{F}_{k_n}\right)\right]\cr
  &=\mathbb{E}\!\left[\mathbb{P}\!\left(\left|\overline {\mathcal{Z}}^{Z_{k_n}}_{n-k_n}(\sqrt n\,A)-R_n\right|>\delta\,\middle|\,\mathcal{F}_{k_n}\right)\right]\cr
  &\leq\mathbb{E}\!\left[C_1e^{-C_2\delta^2|Z_{k_n}|}\right]\cr
  &=C_1e^{-C_2\delta^2b^{k_n}},\nonumber
\end{align}
where the second equality follows from the branching property. This, together with \eqref{eq:kn-properties}, concludes the lemma for $\delta>0$ sufficiently small. Since $\mathbb{P}\!(|\overline Z_n(\sqrt n\,A)-R_n|>\delta)$ is decreasing with respect to $\delta$, the lemma holds with any $\delta>0$.
\end{proof}

By \eqref{5j7u7ugt5} and \eqref{45t46yy5t}, $H_{Y_n}(\xi)=\frac{S_{\xi|k_n}}{\sqrt{m_n}}$. Thus, similar to the argument of \eqref{eq:path-L1}, we have
\begin{equation}
  \Psi_A(Y_n)
  = \int_{\partial\mathbb{T}_b}g_A(H_{Y_n}(\xi))\,\mathbf m_b(\mathrm{d}\xi)=
  b^{-k_n}\sum_{|u|=k_n}
  \nu\!\left(A-\frac{S_u}{\sqrt{m_n}}\right).\nonumber
\end{equation}
The following lemma shows that $\{\Psi_A(Y_n)\}_{n\geq1}$ can be a replacement of $\{R_n\}_{n\geq1}$, which implies that they are exponentially
equivalent.
\begin{lemma}
\label{lem:kernel-replacement}
Suppose that the branching random walk $\{Z_{n}\}_{n\geq0}$ is defined on the probability space $(\Omega,\mathcal{F},\mathbb{P})$. Then,
\[
  \lim_{n\to\infty}\sup_{\omega\in \Omega}|R_n-\Psi_A(Y_n)|=0.
\]
Consequently, \(\{R_n\}_{n\geq1}\) and \(\{\Psi_A(Y_n)\}_{n\geq1}\) are exponentially equivalent.
\end{lemma}

\begin{proof}
Put \(\rho_n:=\sqrt{n/(n-k_n)}\). Since $\lim_{n\to\infty}\rho_n=1$, for any fixed \(\ell>1\), one has
\(\rho_n\in[\ell^{-1},\ell]\) for all sufficiently large \(n\). Therefore,
\begin{align}\label{e5yhyt74}
  &|R_n-\Psi_A(Y_n)|\cr
  &=\left|\frac{1}{|Z_{k_n}|}\sum_{u\in Z_{k_n}}
  \nu_{n-k_n}\!\left(\sqrt n\,A-S_u\right)- b^{-k_n}\sum_{u\in Z_{k_n}}
  \nu\!\left(A-\frac{S_u}{\sqrt{n-k_n}}\right)\right|\cr
  &\leq
  \sup_{z\in\mathbb{R}}
  \left|
    \nu_{n-k_n}\!\left(\sqrt{n-k_n}\,(\rho_nA-z)\right)
    -\nu(\rho_nA-z)
  \right|\cr
  &+\sup_{z\in\mathbb{R}}|\nu(\rho_nA-z)-\nu(A-z)|,
\end{align}
which tends to $0$ by Lemma~\ref{lem:uniform-CLT}. Moreover, observe that the bound in \eqref{e5yhyt74} does not depend
on $\omega\in\Omega$. Thus, $\lim_{n\to\infty}\sup_{\omega\in \Omega}|R_n-\Psi_A(Y_n)|=0$. Hence, for every \(\delta>0\), there exists a deterministic integer $N$ such that for $n>N$,
\[
  \mathbb{P}(|R_n-\Psi_A(Y_n)|>\delta)=0.
\]
Thus, \(\{R_n\}_{n\geq1}\) and \(\{\Psi_A(Y_n)\}_{n\geq1}\) are exponentially equivalent.
\end{proof}
Now, we are ready to prove Theorem~\ref{thm:empirical-LDP}.
\begin{proof}[Proof of Theorem~\ref{thm:empirical-LDP}]

Recall that $\{Z_n(\sqrt n\,A)\}_{n\geq1}$ and $\{R_n\}_{n\geq1}$ are exponentially
equivalent (see Lemma~\ref{lem:conditional-concentration}), and $\{R_n\}_{n\geq1}$ and \(\{\Psi_A(Y_n)\}_{n\geq1}\) are exponentially
equivalent (see Lemma~\ref{lem:kernel-replacement}). This, together with the fact that
\begin{align*}
  &\mathbb{P}\!\left(
    |\overline Z_n(\sqrt n\,A)-\Psi_A(Y_n)|>\delta
  \right)\cr
  &\leq
  \mathbb{P}\!\left(
    |\overline Z_n(\sqrt n\,A)-R_n|>\delta/2
  \right)+
  \mathbb{P}\!\left(|R_n-\Psi_A(Y_n)|>\delta/2\right),
\end{align*}
yields that
\(\{\overline Z_n(\sqrt n\,A)\}_{n\geq1}\) and \(\{\Psi_A(Y_n)\}_{n\geq1}\) are exponentially
equivalent with speed $n^{\alpha/2}$. By Lemma~\ref{lem:exp-equivalence-transfer}, it suffices to consider the LDP for \(\{\Psi_A(Y_n)\}_{n\geq1}\). In Theorem~\ref{thm:tree-field-LDP}, we have proved that  \(\{Y_n\}_{n\geq1}\) satisfy a full LDP with speed
\(n^{\alpha/2}\) and good rate function \(\mathcal{I_{\alpha}}\). Since $\Psi_A$ is continuous (see Lemma \ref{lem:Psi-Lipschitz}), by Lemma \ref{lem:contraction}, we get that $\{\Psi_A(Y_n)\}_{n\geq1}$
satisfy a full LDP with speed \(n^{\alpha/2}\) and good rate function
$$
\inf\{\mathcal{I_{\alpha}}(h):\Psi_A(h)=p, h\in\mathcal{H}_b\}=Q_A(p).
 $$
The infimum is attained by Proposition~\ref{prop:cost-bounds} (iii).
\end{proof}

\subsection{Properties of the rate function \texorpdfstring{\(Q_A\)}{QA}}\label{dtgxbgfdgge}

Recall that
\[
  m_A=\inf_{x\in\mathbb{R}}g_A(x),
  \qquad
  M_A=\sup_{x\in\mathbb{R}}g_A(x),
  \qquad
  D_A=\{g_A(x):x\in\mathbb{R}\}
\]
and
$$
 Q_A(q)
  =
  \inf\Big\{\lambda\sum_{v\neq\varnothing}|h_v|^\alpha:h\in\mathcal{H}_b,\ \Psi_A(h)=q\Big\}.
  $$
We point out that if
\(q\in D_A\), then the following setting implies that $\Psi_A(h)=q$ has uncountably many feasible fields. Since \(q\in D_A\), we can choose $x$ such that $g_A(x)=q.$ For every \(s\in\mathbb{R}\),
set the field
\begin{align}\label{rsgfdgty5y3}
  h_v^{(s)}:=
  \begin{cases}
    s,&|v|=1,\\
    x-s,&|v|=2,\\
    0,&|v|\geq3,
  \end{cases}
\end{align}
which belongs to \(\mathcal{H}_b\). Hence, \(H_{h^{(s)}}\equiv x\) and
\(\Psi_A(h^{(s)})=q\) for any \(s\in\mathbb{R}\). In the following, we present several properties of $Q_A$.

\begin{proposition}
\label{prop:QA-regularity}
The rate function \(Q_A\) has the following properties.\\
(i) \(\{q\in[0,1]:Q_A(q)\in[0,\infty)\}=D_A\), and \(Q_A(q)=0\) if and only if
  \(q=\nu(A)\).\\
  (ii) \(Q_A\) is strictly decreasing on \(D_A\cap[m_A,\nu(A)]\) and strictly
  increasing on  \(D_A\cap[\nu(A),M_A]\).\\
   (iii) \(Q_A\) is lower semicontinuous on $[0,1]$. Hence, it is right-continuous on \(D_A\cap[m_A,\nu(A))\) and left-continuous on \(D_A\cap(\nu(A),M_A]\).\\
  (iv) The discontinuities of \(Q_A\) are jumps and form an at most countable set.
\end{proposition}
\begin{proof}
For every \(q\in D_A\), by \eqref{rsgfdgty5y3}, we have $Q_A(q)\leq \lambda b|s|^{\alpha}+\lambda b^2|x-s|^{\alpha}<\infty$. Thus,
\(D_A\subset\{q\in[0,1]:Q_A(q)\in[0,\infty)\}\). It remains to show $D_A\supset\{q\in[0,1]:Q_A(q)\in[0,\infty)\}$. Suppose that $Q_A(q)<\infty$. Then, by Proposition~\ref{prop:cost-bounds} (iii), there exists $h$ such that
$$q=\Psi_A(h)= \int_{\partial\mathbb{T}_b}g_A(H_h(\xi))\,\mathbf m_b(\mathrm{d}\xi)\in D_A,$$
where the last inclusion follows because \(D_A\) is an interval. Hence, \(\{q\in[0,1]:Q_A(q)\in[0,\infty)\}=D_A\).
Note that we can choose \(h=\mathbf{0}\)
such that \(\Psi_A(h)=\Psi_A(\mathbf{0})=\nu(A)\).
Thus, $Q_A(\nu(A))=0$. Conversely, if $Q_A(q)=0$, then
by Proposition~\ref{prop:cost-bounds} (iii), there exists $h$
such that $\Psi_A(h)=q$ and $\sum_{v\neq\varnothing}|h_v|^\alpha=0$,
which means $h=\mathbf{0}.$ Thus, $q=\Psi_A(h)=\Psi_A(\mathbf{0})=\nu(A)$. We have completed the proof of (i).
\par

We proceed to deal with (ii). Without loss of generality, we only show
$Q_A$ is strictly increasing on \(D_A\cap[\nu(A),M_A]\). Let \(q_1<q_2\) and $q_1,~q_2\in D_A\cap[\nu(A),M_A]$.
By Proposition~\ref{prop:cost-bounds} (iii), we can choose \(h_2\) attaining \(Q_A(q_2)\).
The continuous function \(l(t)=\Psi_A(th_2),~t\geq0\) satisfies that $l(0)=\nu(A)$ and $l(1)=q_2$. Hence, there exists \(t\in[0,1)\) such that
\(\Psi_A(th_2)=q_1\). By the definition of $Q_A$,
\[
  Q_A(q_1)\leq\mathcal{I_{\alpha}}(th_2)=t^\alpha Q_A(q_2)<Q_A(q_2).
\]
This proves (ii).
\par
From Proposition \ref{prop:cost-bounds} (ii), $Q_A$ is lower semicontinuous. This, combined with the monotonicity proved in (ii), concludes the one-sided continuity. Thus, (iii) holds.
\par
Next, we are going to deal with (iv). Since a monotone function has at most countably many discontinuities, by (iii), it suffices to show that $q=\nu(A)$ cannot be a removable discontinuity point of $Q_A$. To this end, we show that either $\lim_{q\to\nu(A)+}Q_A(q)=Q_A(\nu(A))$ or $\lim_{q\to\nu(A)-}Q_A(q)=Q_A(\nu(A))$ holds. By Assumption \ref{45twzg5yy43} (4), \(g_A\) is nonconstant and real analytic. Thus, there exist \(k\geq1\) and \(c\neq0\) such that
$$\lim_{x\to0}\frac{g_A(x)-g_A(0)}{x^k}=c;$$
see \cite[Sections~1.1--1.2]{KrantzParks}. We only consider $c>0$ here, for $c<0$, the proofs are similar. Then there exists a non-negative sequence $\{x_n\}_{n\geq1}$ such that $\lim_{n\to\infty}x_n=0$, $g_A(x_n)>g_A(0),~n\geq1$ and $\lim_{n\to\infty}g_A(x_n)= \nu(A)=g_A(0)$. Similar to \eqref{tryj66tyj}, we can get \(Q_A(g_A(x))\leq\lambda b|x|^\alpha\) for every \(x\in\mathbb R\). Thus, it follows that
$$0\leq\lim_{q\to\nu(A)+}Q_A(q)=\lim_{n\to\infty}Q_A(g_A(x_n))\leq \lim_{n\to\infty}\lambda b|x_n|^\alpha=0,$$
where the second equality follows from the monotonicity of $Q_A$ on \(D_A\cap[\nu(A),M_A]\). Hence, $\lim_{q\to\nu(A)+}Q_A(q)=0=Q_A(\nu(A)).$
\end{proof}
\begin{remark}
  We also point that \(Q_A\) may not convex. To see this, set \(A=(-\infty,0]\) and \(0<\alpha<1\). Recall that $\Phi$ is the distribution function of a standard normal random
variable. Let $\Phi^{-1}$ be the inverse function of $\Phi$. For $0.5<p<0.5+{(\Phi(\sqrt{1-\alpha})-0.5)/}{b}$, we have \(0<t^2<1-\alpha\), where $t=\Phi^{-1}\left(bp-(b-1)/2\right)$. Moreover, since \(\Phi(0)=0.5\), we have
\(0<s:=b(p-\Phi(0))<\Phi(\sqrt{1-\alpha})-0.5<0.5=1-\Phi(0)\). Thus, the conditions of Example \ref{5tre6yy6yu} are satisfied. Hence,
$
Q_A(p)=\lambda t^\alpha.
$
Since
\[
Q_A''(p)
=
\frac{\lambda\alpha b^2t^{\alpha-2}}
     {\varphi(t)^2}
\left(t^2+\alpha-1\right)
<0,
\]
\(Q_A\) is strictly concave on $(0.5,0.5+{(\Phi(\sqrt{1-\alpha})-0.5)/}{b})$. Therefore $Q_A$ is not convex in general.
\end{remark}

\subsection{Proof of Theorem~\ref{thm:main-one-sided}}
Now, we are ready to prove Theorem~\ref{thm:main-one-sided}, which is a direct application of Theorem \ref{thm:empirical-LDP}.
\begin{proof}[Proof of Theorem~\ref{thm:main-one-sided}]
In Theorem \ref{thm:empirical-LDP}, set $\Gamma=[p,1]$. Then, it follows that
\begin{align*}
  -Q_A(p+)=-\inf_{q\in(p,1]}Q_A(q)
  &\leq
  \liminf_{n\to\infty}n^{-\alpha/2}
  \log\mathbb{P}\!\left(\overline Z_n(\sqrt n\,A)\geq p\right)\\
  &\leq
  \limsup_{n\to\infty}n^{-\alpha/2}
  \log\mathbb{P}\!\left(\overline Z_n(\sqrt n\,A)\geq p\right)\\
  &\leq
  -\inf_{q\in[p,1]}Q_A(q)=-Q_A(p),
\end{align*}
where the first and last equalities follow from the fact that $Q_A$ is increasing on $[\nu(A),M_A)$ (see Proposition \ref{prop:QA-regularity} (ii)).
\end{proof}

\section{Continuity criteria and examples}\label{5yehye6t5t}
Theorem \ref{thm:main-one-sided} involves the two constants \(Q_A(p)\) and \(Q_A(p+)\). This section characterizes when they are equal, and gives sufficient conditions for it. Moreover, we also present both continuity examples and counter-examples for $Q_A$.

\subsection{Necessary and sufficient criteria}
In this section, we shall give several criteria for $Q_A(p)=Q_A(p+)$. For $h\in\mathcal{H}_b$, put
\[
  \|h\|_{\ell^\alpha}
  :=\left(\sum_{v\neq\varnothing}|h_v|^\alpha\right)^{1/\alpha},
\]
with $\|h\|_{\ell^\alpha}=\infty$ when the series diverges. $\|\cdot\|_{\ell^\alpha}$ is a quasi-norm when \(0<\alpha<1\) and a norm when $\alpha\geq1$. For \(p\in(\nu(A),M_A)\), define the sets of optimal minimizers of $Q_A(p)$ by
\[
  \mathcal O_A(p)
  :=\{h\in\mathcal{H}_b:\Psi_A(h)=p,\ \mathcal{I_{\alpha}}(h)=Q_A(p)\},
\]
and the finite tree energy fields by
\[
  \mathcal U_A(p)
  :=\{h\in\mathcal{H}_b:\mathcal{I_{\alpha}}(h)<\infty,\ \Psi_A(h)>p\}.
\]
Define the distance between $\mathcal O_A(p)$ and $\mathcal U_A(p)$ by
\begin{equation}\label{eq:escape-radius}
  \rho_{\alpha,A}(p)
  :=\inf\left\{
    \|\widetilde h-h\|_{\ell^\alpha}:
    h\in\mathcal O_A(p),\ \widetilde h\in\mathcal U_A(p)
  \right\}.
\end{equation}
Note that \(\mathcal O_A(p)\) and $\mathcal U_A(p)$ are nonempty by
Proposition~\ref{prop:cost-bounds} (iii).

\begin{theorem}[Continuity criterion]
\label{thm:optimizer-escape}
For every \(\alpha>0\) and \(p\in(\nu(A),M_A)\), the following assertions are equivalent:\\
(i) \(Q_A(p+)=Q_A(p)\);\\
(ii) \(\rho_{\alpha,A}(p)=0\);\\
(iii) there exist \(h_p\in\mathcal O_A(p)\) and \(\{h_n\}_{n\geq1}\subset\mathcal U_A(p)\) such that
  \[
    \lim_{n\to\infty}\|h_n-h_p\|_{\ell^\alpha}=0.
  \]
\end{theorem}

\begin{proof}
We begin with \textup{(ii)}\(\Rightarrow\)\textup{(i)}. Because \(Q_A\) is increasing on \((\nu(A),M_A)\), by \eqref{eq:QA},
\begin{equation}\label{eq:open-cost-levelwise}
  Q_A(p+)
  =\inf_{q>p}Q_A(q)
  =\inf\{\mathcal{I_{\alpha}}(h): h\in\mathcal{H}_b, \Psi_A(h)>p\}.
\end{equation}
Since (ii) holds, we can choose
\(h_n\in\mathcal O_A(p)\) and
\(\widetilde h_n\in\mathcal U_A(p)\) such that
\(\lim_{n\to\infty}\|\widetilde h_n-h_n\|_{\ell^\alpha}=0\). The map
\(h\mapsto\lambda\|h\|_{\ell^\alpha}^\alpha\) is continuous in $\|\cdot\|_{\ell^\alpha}$
topology for every \(\alpha>0\) (by subadditivity when
\(0<\alpha\leq1\), and by norm continuity when \(\alpha>1\)). Hence,
\[
  \lim_{n\to\infty}\mathcal{I_{\alpha}}(\widetilde h_n)-\mathcal{I_{\alpha}}(h_n)=0.
\]
Since
\(\mathcal{I_{\alpha}}(h_n)=Q_A(p)\) and \(\Psi_A(\widetilde h_n)>p\),
\[
  Q_A(p+)
  =\inf_{\substack{h\in\mathcal{H}_b\\\Psi_A(h)>p}}\mathcal{I_{\alpha}}(h)
  \leq\lim_{n\to\infty}\mathcal{I_{\alpha}}(\widetilde h_n)
  =Q_A(p),
\]
where the first equality follows from \eqref{eq:open-cost-levelwise}. This completes the proof of \textup{(ii)}\(\Rightarrow\)\textup{(i)}.
\par
The implication \textup{(iii)}\(\Rightarrow\)\textup{(ii)} follows directly from the definition of $\rho_{\alpha,A}(p)$; see \eqref{eq:escape-radius}. So, it remains to show \textup{(i)}\(\Rightarrow\)\textup{(iii)}. Since (i)
holds, choose \(q_n\downarrow p\) and minimizer \(h_n\) of \(Q_A(q_n)\)
such that
\[
  \Psi_A(h_n)=q_n,
  \quad
  \mathcal{I_{\alpha}}(h_n)=Q_A(q_n)\to Q_A(p)~\text{as}~n\to\infty.
\]
Similar to those arguments used in proving \eqref{4tgty6gr4tyr4}, there exists $h_p\in\mathcal{H}_b$ such that along a subsequence $\{n_j\}_{j\geq1}$,
\begin{align}\label{4tgty6yr4}
  \lim_{j\to\infty}\|h_{n_j}-h_p\|_b=0.
\end{align}
Continuity of \(\Psi_A\) (see Lemma \ref{lem:Psi-Lipschitz}) and lower
semicontinuity of \(\mathcal{I_{\alpha}}\) (see Proposition \ref{prop:energy-good}) give
\[
  \Psi_A(h_p)=p,
  \qquad
  Q_A(p)\leq\mathcal{I_{\alpha}}(h_p)
  \leq\liminf_{j\to\infty}\mathcal{I_{\alpha}}(h_{n_j})=Q_A(p),
\]
which implies \(h_p\in\mathcal O_A(p)\). Since $\mathcal{I_{\alpha}}(h_n)=Q_A(q_n)\to Q_A(p)=\mathcal{I_{\alpha}}(h_p)$, it follows that
\begin{equation}\label{eq:energy-sum-convergence}
  \lim_{j\to\infty}\sum_{v\neq\varnothing}|(h_{n_j})_v|^\alpha=
  \sum_{v\neq\varnothing}|(h_p)_v|^\alpha=Q_A(p)/\lambda<\infty.
\end{equation}
For any $M\geq1$, \eqref{4tgty6yr4} implies the coordinatewise convergence on the
finite set \(\{v:|v|\leq M\}\), which entails that
\begin{align}\label{ryhythy5}
  \lim_{j\to\infty}\sum_{|v|\leq M}|(h_{n_j})_v|^\alpha=\sum_{|v|\leq M}|(h_p)_v|^\alpha,~\lim_{j\to\infty}\sum_{|v|\leq M}|(h_{n_j})_v-(h_{p})_v|^\alpha=0.
\end{align}
This, combined with \eqref{eq:energy-sum-convergence}, concludes that
\[
  \lim_{j\to\infty}\sum_{|v|> M}|(h_{n_j})_v|^\alpha=\sum_{|v|> M}|(h_p)_v|^\alpha.
  \]
Thus, given any \(\eta>0\), we can choose \(M\) large enough such that
\begin{align}
&\limsup_{j\to\infty}\sum_{|v|>M}|(h_{n_j})_v|^\alpha\leq2\sum_{|v|>M}|(h_p)_v|^\alpha\leq\eta.\nonumber
\end{align}
Since $|x-y|^\alpha\leq 2^{(\alpha-1)^+}(|x|^\alpha+|y|^\alpha)$ for any $\alpha>0$, above yields that
\begin{align}\label{5gbrh6hyh6}
  &\limsup_{j\to\infty}
  \sum_{v\neq\varnothing}|(h_{n_j})_v-(h_p)_v|^\alpha\cr
 &\leq  \limsup_{j\to\infty}
  \sum_{|v|\leq M}|(h_{n_j})_v-(h_p)_v|^\alpha +\limsup_{j\to\infty}
  \sum_{|v|>M}|(h_{n_j})_v-(h_p)_v|^\alpha \cr
   &\leq \limsup_{j\to\infty}
  2^{(\alpha-1)^+}\sum_{|v|>M}\left[|(h_{n_j})_v|^\alpha+|(h_p)_v|^\alpha\right]\leq2^{1+(\alpha-1)^+}\eta,
\end{align}
where the second inequality follows from the second equality of \eqref{ryhythy5}.
Letting \(\eta\downarrow0\) in \eqref{5gbrh6hyh6} gives
\(\lim_{j\to\infty}\|h_{n_j}-h_p\|_{\ell^\alpha}=0\). Since
\(\Psi_A(h_{n_j})=q_{n_j}>p\), we have \(\{h_{n_j}\}_{j\geq1}\subset\mathcal U_A(p)\). This proves \textup{(iii)}.
\end{proof}
\begin{remark}
The topology in Theorem~\ref{thm:optimizer-escape} is dictated by the rate
function itself: convergence in \(\ell^\alpha\) makes the energies converge,
whereas convergence in the weaker tree-space norm \(\|\cdot\|_b\) need
not do so.
\end{remark}

\subsection{A sufficient criterion: absence of suboptimal local maxima}
The criterion in Theorem \ref{thm:optimizer-escape} is stated in terms
of the minimizers of $Q_A(p)$, which seems abstract.
The following condition involves only \(g_A\) and is easier to verify.

\begin{proposition}
\label{cor:no-suboptimal-maxima}
If every local maximizer of \(g_A\) is global, then
\[
  Q_A(q+)=Q_A(q)
  \qquad\text{for every }q\in(\nu(A),M_A).
\]
\end{proposition}

\begin{proof}
Fix \(q\in(\nu(A),M_A)\). Throughout the proof, local maximality of $\Psi_A$ is understood in the \(\ell^\alpha\)-topology. We first claim that
\begin{equation}
Q_A(q+)=Q_A(q)\Longleftrightarrow
\exists\,h\in \mathcal O_A(q)\text{ that is not a local maximizer of }\Psi_A.
\label{eq:local-max-equivalence}
\end{equation}
We first show the $``\Rightarrow"$ direction. Suppose that $Q_A(q+)=Q_A(q)$. Then, by Theorem~\ref{thm:optimizer-escape}, there exist $h_q\in\mathcal O_A(q)$ and \(\{h_n\}_{n\geq1}\subset\mathcal U_A(q)\) such that
$\lim_{n\to\infty}\|h_n-h_q\|_{\ell^\alpha}=0$. Since $\Psi_A(h_n)>q$ and $\Psi_A(h_q)=q$, $h_q$ is not a local maximizer of $\Psi_A$, which means $``\Rightarrow"$ holds. We proceed to deal with the $``\Leftarrow"$ direction. Suppose that \(h\in \mathcal O_A(q)\) is not a local maximizer of \(\Psi_A\). Then, there exists \(\{h_n\}_{n\geq1}\subset\mathcal H_b\) such that \(\lim_{n\to\infty}\|h_n-h\|_{\ell^\alpha}=0\) and \(\Psi_A(h_n)>\Psi_A(h)=q,~n\geq1\). If \(0<\alpha\leq1\), then
\begin{equation*}
\mathcal{I_{\alpha}}(h_n)=\lambda\sum_{v\neq\varnothing}|(h_n)_v|^\alpha
\leq \mathcal{I_{\alpha}}(h)+\lambda \|h_n-h\|^{\alpha}_{\ell^\alpha}
<\infty,
\end{equation*}
where the last inequality follows from  $\mathcal{I_{\alpha}}(h)=Q_A(q)<\infty$ for \(h\in \mathcal O_A(q)\).
If \(\alpha>1\), then $\|\cdot\|_{\ell^\alpha}$ is a norm. Thus,
\begin{equation*}
\|h_n\|_{\ell^\alpha}
\leq
\|h\|_{\ell^\alpha}
+
\|h_n-h\|_{\ell^\alpha}
<\infty,
\end{equation*}
which means \(I_\alpha(h_n)<\infty\). Hence, for any $\alpha>0$, we have \(\{h_n\}_{n\geq1}\subset\mathcal U_A(q)\). So, by Theorem~\ref{thm:optimizer-escape}, we have $Q_A(q+)=Q_A(q)$.
\par
Now, we prove Proposition \ref{cor:no-suboptimal-maxima} by contradiction. Suppose that \(Q_A(q+)\neq Q_A(q)\). Then, by \eqref{eq:local-max-equivalence}, any \(h_q\in \mathcal O_A(q)\) is a local maximizer of \(\Psi_A\). Hence, there exists
\(\varepsilon_0=\varepsilon_0(h_q)>0\) such that
\begin{equation}\label{eq:local-max-Psi}
  \Psi_A(\widetilde h)\leq\Psi_A(h_q)=q
  \quad\text{for}
  \ \|\widetilde h-h_q\|_{\ell^\alpha}<\varepsilon_0.
\end{equation}
For \(u\in\mathbb{T}_b^\circ\), define the field
\(e^{(u)}=(e^{(u)}_v)_{v\in\mathbb{T}_b^\circ}\) by
\[
  e^{(u)}_v:=\mathbf{1}_{\{v=u\}},
  \qquad v\in\mathbb{T}_b^\circ.
\]
Then
\[
  H_{h_q+te^{(u)}}(\xi)
  =H_{h_q}(\xi)+t\mathbf{1}_{[u]}(\xi),
  \qquad
  \|te^{(u)}\|_{\ell^\alpha}=|t|.
\]
Set $\widetilde h=h_q+te^{(u)}$ in \eqref{eq:local-max-Psi}. Then, \eqref{eq:local-max-Psi} gives for every
\(|t|<\varepsilon_0\),
\begin{align}\label{5thtr6ye3}
  0
  &\geq \Psi_A(h_q+te^{(u)})-\Psi_A(h_q)\notag\\
  &=\int_{\partial\mathbb{T}_b}
  \bigl[g_A(H_{h_q}(\xi)+t\mathbf{1}_{[u]}(\xi))
         -g_A(H_{h_q}(\xi))\bigr]
  \,\mathbf m_b(\mathrm{d}\xi)\notag\\
  &=\int_{[u]}
  \bigl[g_A(H_{h_q}(\xi)+t)-g_A(H_{h_q}(\xi))\bigr]
  \,\mathbf m_b(\mathrm{d}\xi).
\end{align}
The last equality follows because
\(H_{h_q+te^{(u)}}=H_{h_q}\) on \([u]^c\).
Define
\[
  F_t(\xi)
  :=g_A(H_{h_q}(\xi)+t)-g_A(H_{h_q}(\xi)).
\]
\par

Next, we are going to show that $F_t\leq0$ $\mathbf m_b\text{-a.s.}$ for any given $t\in(-\varepsilon_0,\varepsilon_0)$. Define
\begin{align*}
\mathcal A
&:=
\{\varnothing\}
\cup
\left\{
\bigcup_{j=1}^{k}[u_j]:
k\geq1,\
u_1,\ldots,u_k\in\mathbb T_b^\circ
\right\}.
\end{align*}
Then \(\mathcal A\) is an algebra and $\sigma(\mathcal A)=\mathcal F_{\partial}$; see \eqref{54gt24g5gt}. For every \(B\in\mathcal A\), there exist \(N\geq1\) and
\(D\subset\{u\in\mathbb T_b:|u|=N\}\) such that
\begin{equation*}
B=\bigsqcup_{u\in D}[u],
\end{equation*}
where $\bigsqcup$ stands for the disjoint union. Consequently, by \eqref{5thtr6ye3},
\begin{equation*}
\int_BF_t(\xi)\mathbf m_b(\mathrm{d}\xi)
=
\sum_{u\in D}\int_{[u]}F_t(\xi)\mathbf m_b(\mathrm{d}\xi)
\leq0,
\qquad B\in\mathcal A.
\end{equation*}
Note that $\{B\in\mathcal F_{\partial}:\int_BF_t(\xi)\mathbf m_b(\mathrm{d}\xi)
\leq0\}\supset\mathcal A$ and is a monotone class. By the monotone class theorem, we obtain
\begin{equation*}
\int_BF_t(\xi)\mathbf m_b(\mathrm{d}\xi)
\leq0,
\quad B\in\mathcal F_{\partial},
\end{equation*}
which implies that $\mathbf m_b(\{\xi\in \partial\mathbb T_b:F_t(\xi)\leq 0\})=1.$
\par
Above yields that
$$\mathbf m_b\Big(\cap_{|t|<\varepsilon_0~\text{and}~t~\text{is~rational} }\{\xi\in \partial\mathbb T_b:F_t(\xi)\leq 0\}\Big)=1.$$
By the
continuity of \(g_A\), it follows that for almost every \(\xi\) and $|t|<\varepsilon_0$,
\[
  g_A(H_{h_q}(\xi)+t)\leq g_A(H_{h_q}(\xi)).
\]
In other words, for almost every \(\xi\), the real number
\(x=H_{h_q}(\xi)\) is a local maximizer of the same deterministic function
\(g_A\). Since we assume every local maximizer of $g_A$ is global in Proposition \ref{cor:no-suboptimal-maxima}, we have
\[
  q=\Psi_A(h_q)
  =\int_{\partial\mathbb{T}_b}g_A(H_{h_q}(\xi))\,\mathbf m_b(\mathrm{d}\xi)
  =M_A.
\]
This contradicts our setting \(q<M_A\). Therefore the supposition
\(Q_A(q+)\neq Q_A(q)\) is false, and \(Q_A(q+)=Q_A(q)\).
\end{proof}
\subsection{Examples}
In this subsection, we present three classes of sets such that $Q_A(p)=Q_A(p+)$ for every \(p\in(\nu(A),M_A)\) and a counter-example.
\begin{example}
\label{cor:concrete-continuity}
For every \(p\in(\nu(A),M_A)\), one has \(Q_A(p+)=Q_A(p)\) in each of the following cases:\\
  (i) \(A\) is a single interval, including a bounded interval or a half-line;\\
  (ii) \(A=(-\infty,a]\cup[b,\infty)\) for any \(a<b\);\\
  (iii) up to translation,
  \[A=[-b,-a]\cup[a,b],
    \qquad b>a>0.\]
\end{example}
\begin{proof}
We begin with (i). For a bounded interval \(A=[a,b]\),
\[
  g_A'(x)=\varphi(a-x)-\varphi(b-x)=\frac{1}{\sqrt{2\pi}}\left[e^{-(a-x)^2/2}-e^{-(b-x)^2/2}\right].
\]
Thus, \(g_A'>0\) on \(( -\infty,(a+b)/2)\) and
\(g_A'<0\) on \(((a+b)/2,\infty)\). This means \(g_A\) has a unique local maximum, which is also global. For a half-line, since \(g_A\) is strictly monotone on $\mathbb{R}$, \(g_A\) does not have a local maximum. So, by Proposition \ref{cor:no-suboptimal-maxima}, we have
\(Q_A(p+)=Q_A(p)\) for \(p\in(\nu(A),M_A)\).
\par
In case \textup{(ii)},
\[
  g_A(x)=1-[\Phi(b-x)-\Phi(a-x)].
\]
The midpoint \(x=(a+b)/2\) is the unique local minimum point, and there is no local
maximum in \(\mathbb{R}\). Thus the hypothesis of
Proposition~\ref{cor:no-suboptimal-maxima} holds vacuously.
\par
We proceed to deal with (iii). Note that translation only shifts the graph of \(g_A\) and therefore preserves whether every local maximizer is global. So, it suffices to take
  \[
    A=[-b,-a]\cup[a,b],
    \qquad 0<a<b.
\]
By symmetry of the Gaussian density $\varphi$, \(g_A\) is even and \(g_A'\) is odd. Thus, we only consider $g_A$ on $(0,\infty)$. For \(x>0\),
\begin{align}\label{54hgdhrtyt2}
R_s(x)&:=\varphi(s-x)-\varphi(s+x)>0{~\text{for}~s>0},\cr
  g_A'(x)&=R_a(x)-R_b(x)=R_b(x)\left(\frac{R_a(x)}{R_b(x)}-1\right).
\end{align}
Using \(\varphi'(y)=-y\varphi(y)\), a direct differentiation gives
\[
  \frac{\mathrm{d}}{\mathrm{d} x}\log\frac{R_a(x)}{R_b(x)}=a~\text{coth}(ax)-b~\text{coth}(bx)<0,
  \qquad x>0,
\]
where $a~\text{coth}(ax)=a{(e^{ax}+e^{-ax})/}{(e^{ax}-e^{-ax})}$ is increasing with respect to $a$ for fixed $x>0$.
Thus, \(R_a/R_b\) is strictly decreasing on $(0,\infty)$. Moreover, since $0<a<b$, we have
$$
\lim_{x\to+\infty}\frac{R_a(x)}{R_b(x)}=\lim_{x\to+\infty}\frac{e^{-\frac{(a-x)^2}{2}}-e^{-\frac{(a+x)^2}{2}}
}{e^{-\frac{(b-x)^2}{2}}-e^{-\frac{(b+x)^2}{2}}}=\lim_{x\to+\infty}\frac{e^{\frac{(b-x)^2}{2}-\frac{(a-x)^2}{2}}-e^{\frac{(b-x)^2}{2}-\frac{(a+x)^2}{2}}
}{1-e^{\frac{(b-x)^2}{2}-\frac{(b+x)^2}{2}}}=0.
$$
This, combined with \(R_a/R_b>0\) and the monotonicity of \(R_a/R_b\), implies that $R_a/R_b$ either stays below $1$ (when $\lim_{x\to0+}R_a(x)/R_b(x)
\leq1$) or crosses the level $1$ exactly once (when $\lim_{x\to0+}R_a(x)/R_b(x)>1$). Consequently, by \eqref{54hgdhrtyt2}, \(g_A'\) is either negative throughout
\((0,\infty)\), or
changes sign from positive to negative. Hence \(g_A\)
is either strictly decreasing on \((0,\infty)\), or first increases and
then decreases. Since \(g_A\) is even, its local maximizers are therefore
either \(0\), or one symmetric pair; in either case they are global.
The result follows from Proposition~\ref{cor:no-suboptimal-maxima}.
\end{proof}

The preceding example gives broad classes for which \(Q_A\) is right-continuous at $p$. The next
result gives a counter-example.

\begin{example}
\label{prop:genuine-jump}
Assume \(0<\alpha\leq1\). Fix \(0<a<c\). For
\(\ell>c\), define
\[
  A_\ell=[a,c]\cup[\ell,\infty).
\]
Then, for all sufficiently large \(\ell\), there exists
\(p_\ell\in(\nu(A_\ell),M_{A_\ell})\) such that
\[
  Q_{A_\ell}(p_\ell)<Q_{A_\ell}(p_\ell+).
\]
\end{example}

\begin{proof}
Write
\[
  g_\ell(x):=g_{A_\ell}(x)
  =\nu(A_{\ell}-x)=\Phi(c-x)-\Phi(a-x)+\Phi(x-\ell),
\]
and denote the contribution of the bounded interval by
\[
  f(x):=\Phi(c-x)-\Phi(a-x).
\]
Let $\mu:=(a+c)/2>0.$ Observe that \(f'>0\) on \((-\infty,\mu)\) and \(f'<0\) on
\((\mu,\infty)\). Thus \(f\) has a unique global maximum at \(\mu\).
Moreover, $f'(\mu)=0,~f''(\mu)<0.$ Choose \(\eta>0\) small enough and $\ell>0$ large enough such that \(f''<0\) and $g_\ell''<0$ on
\(I:=[\mu-\eta,\mu+\eta]\). Since $f'<0$ on $(\mu,\infty)$, for large $\ell$, we have
\begin{align}
  g_\ell'(\mu)&=f'(\mu)+\varphi(\mu-\ell)=\varphi(\mu-\ell)>0,\cr
  g_\ell'(\mu+\eta)&=f'(\mu+\eta)+\varphi(\mu+\eta-\ell)<0.\nonumber
\end{align}
 Since \(f\) has a unique global maximum at \(\mu\),
\[
  \delta_\eta
  :=f(\mu)-\sup_{|x-\mu|\geq\eta}f(x)>0.
\]
Since $g_\ell''<0$ on $I$, $g_\ell'$ is decreasing on $I$. Hence, there exists a unique \(x_\ell\in(\mu,\mu+\eta)\) such that \(g_\ell'(x_\ell)=0\). Set \(p_\ell:=g_\ell(x_\ell)\). Note that $g'_\ell>0$ on  $[\mu-\eta,x_\ell)$ and  $g'_\ell<0$ on $(x_\ell,\mu+\eta]$. Hence, \(g_\ell(x)>p_\ell=g_\ell(x_\ell)\) implies \(|x-\mu|\geq \eta\), which means
\[
  \Phi(x-\ell)
  =g_\ell(x)-f(x)
  >p_\ell-(f(\mu)-\delta_\eta)>g_\ell(\mu)-(f(\mu)-\delta_\eta)
  >\delta_\eta.
\]
Thus, if \(g_\ell(x)>p_\ell\), then
\begin{equation}\label{eq:jump-superlevel-location}
  x>\ell+\Phi^{-1}(\delta_\eta),
\end{equation}
where $\Phi^{-1}$ stands for the inverse function of $\Phi$.
\par
Next, we show that $p_\ell\in(\nu(A_\ell),M_{A_\ell})$. Since $g_\ell'>0$ on $[0,x_\ell)$, we have
\[
  \nu(A_\ell)
  =g_\ell(0)
  <g_\ell(x_\ell)
  =p_\ell.
\]
Since
\(\lim_{x\to+\infty}g_\ell(x)=\lim_{x\to+\infty}\nu(A_{\ell}-x)=1\), it follows that
\[
  M_{A_\ell}=1,
  \qquad
  p_\ell\in(\nu(A_\ell),M_{A_\ell}).
\]
\par

By Proposition \ref{prop:cost-bounds} (i), we have
\begin{equation}\label{eq:jump-closed-upper}
  Q_{A_\ell}(p_\ell)\leq\lambda b|x_\ell|^\alpha.
\end{equation}
We proceed to give a lower bound for $Q_{A_\ell}(p_\ell+)$. Notice that
$$
Q_{A_\ell}(p_\ell+)=\inf\Big\{\lambda\sum_{v\neq\varnothing}|h_v|^\alpha:h\in\mathcal{H}_b, \Psi_{A_\ell}(h)=\int_{\partial\mathbb{T}_b}g_{\ell}(H_h(\xi))\,\mathbf m_b(\mathrm{d}\xi)>p_\ell\Big\}.
$$
For any \(h\in\mathcal{H}_b\) satisfying
\(\Psi_{A_\ell}(h)>p_\ell\) and
\(\mathcal{I_{\alpha}}(h)<\infty\), it follows that
\[
  \mathbf m_b(g_\ell(H_h(\xi))>p_\ell)>0.
\]
Since \(0<\alpha\leq1\) and \(\mathcal{I_{\alpha}}(h)<\infty\), it follows that
\[
  \sum_{v\neq\varnothing}|h_v|<\infty.
\]
So, every path series converges absolutely.
Choose a ray \(\xi\in\{\xi:g_\ell(H_h(\xi))>p_\ell\}\). By \eqref{eq:jump-superlevel-location}, for $\ell$ large enough,
\[
  H_h(\xi)>\ell+\Phi^{-1}(\delta_\eta)>0.
\]
Using \(0<\alpha\leq1\), above yields that
\[
  \frac{\mathcal{I_{\alpha}}(h)}{\lambda}
  \geq\sum_{k\geq1}|h_{\xi|k}|^\alpha
  \geq
  \left(\sum_{k\geq1}|h_{\xi|k}|\right)^\alpha
  \geq|H_h(\xi)|^\alpha
  >
  \bigl(\ell+\Phi^{-1}(\delta_\eta)\bigr)^\alpha.
\]
Consequently,
\[
  Q_{A_\ell}(p_\ell+)
  \geq
  \lambda\bigl(\ell+\Phi^{-1}(\delta_\eta)\bigr)^\alpha,
\]
which tends to $+\infty$ as $\ell\to\infty$.
This, combined with the fact that \(\lim_{\ell\to\infty}x_\ell=\mu\) and \eqref{eq:jump-closed-upper},
gives that for all sufficiently large \(\ell\),
\[
  Q_{A_\ell}(p_\ell+)
  \geq
  \lambda\bigl(\ell+\Phi^{-1}(\delta_\eta)\bigr)^\alpha
  >
  \lambda b|x_\ell|^\alpha
  \geq
  Q_{A_\ell}(p_\ell).
\]
\end{proof}

\subsection{Explicit formulas}

\begin{theorem}\label{thm:dimension-reduction}
Suppose that, for some \(\sigma\in\{-1,1\}\), the map
\(t\mapsto g_A(\sigma t)\) is strictly increasing on
 \(\mathbb{R}\). Fix $p\in(\nu(A),M_A)$. Then \(Q_A(p+)=Q_A(p)\), and the following assertions hold.\\
(i) If \(0<\alpha\le1\), then
\begin{equation}\label{eq:first-generation}
  Q_A(p)
  =\lambda\min\left\{
      \sum_{i=1}^b t_i^\alpha:
      t_i\ge0,\quad
      \frac1b\sum_{i=1}^b g_A(\sigma t_i)\ge p
    \right\}.
\end{equation}
The minimum is attained.\\
(ii) If $\alpha\geq1$ and \(t\mapsto g_A(\sigma t)\) is concave on
\([0,\infty)\), then
\[
  Q_A(p)=
  \begin{cases}
    \lambda b I_A(p), & \alpha=1,\\[3pt]
    \lambda\bigl(b^{1/(\alpha-1)}-1\bigr)^{\alpha-1}I_A(p)^\alpha,
      & \alpha>1.
  \end{cases}
\]
\end{theorem}
\begin{proof}
We first give a more restricted representation for $Q_A(p)$. Since \(g_A\) has no local maximizer, Proposition~\ref{cor:no-suboptimal-maxima} gives $Q_A(p+)=Q_A(p).$ The monotonicity assumption and \(g_A(0)=\nu(A)<p\) give
\[
  I_A(p)=\min\{t\ge0:g_A(\sigma t)\ge p\},
  \qquad
  g_A(\sigma I_A(p))=p.
\]
Since $Q_A$ is increasing on $[\nu(A),M_A)$, by the definition of $Q_A$ (see \eqref{eq:QA}), we can write
\begin{align}\label{45gtrh6u}
  Q_A(p)=\inf\{\mathcal{I_{\alpha}}(h):h\in\mathcal{H}_b,\ \Psi_A(h)\ge p\}.
\end{align}
Set $x+:=\max\{x,0\}$. For a finite-energy field \(h\) satisfying \(\Psi_A(h)\ge p\), for ${v\in\mathbb{T}_b^{\circ}}$, put
\[
  y_v:=\sigma h_v,
  \quad
  \widetilde h_v:=\sigma y_v^+.
\]
Since \(\lvert\widetilde h_v\rvert\le\lvert h_v\rvert\) and
\(\widetilde h\in\mathcal{H}_b\), it follows that
\begin{equation}\label{eq:energy-truncation}
  \mathcal{I_{\alpha}}(\widetilde h)
  =\lambda\sum_{v\ne\varnothing}(y_v^+)^\alpha
  \le \lambda\sum_{v\ne\varnothing}\lvert y_v\rvert^\alpha
  =\mathcal{I_{\alpha}}(h).
\end{equation}
Let $y=(y_v)_{v\in\mathbb{T}_b}$ and $y^+=(y^+_v)_{v\in\mathbb{T}_b}$ be fields with $y_{\varnothing}=y^+_{\varnothing}=0$. For $\mathbf m_b$-almost every \(\xi\), we have
\[
  H_{\widetilde h}(\xi)=\sigma H_{y^+}(\xi),
  \qquad
  H_h(\xi)=\sigma H_y(\xi),
  \qquad
  H_{y^+}(\xi)\ge H_y(\xi).
\]
Hence,
\begin{align}
  \Psi_A(\widetilde h)
  &=\int_{\partial\mathbb{T}_b}g_A\bigl(\sigma H_{y^+}(\xi)\bigr)\,\mathbf m_b(\mathrm{d}\xi)\cr
  &\ge\int_{\partial\mathbb{T}_b}g_A\bigl(\sigma H_y(\xi)\bigr)\,\mathbf m_b(\mathrm{d}\xi)=\Psi_A(h)\ge p.\nonumber
\end{align}
This, together with \eqref{eq:energy-truncation} and \eqref{45gtrh6u}, implies that we can write
\begin{equation}\label{eq:positive-y}
 Q_A(p)=\inf\left\{\mathcal{I_{\alpha}}(y):y\in\mathcal{H}_b,~y_v\geq0\ \text{for}\ {v\in\mathbb{T}_b},\ \int_{\partial\mathbb{T}_b}g_A\bigl(\sigma H_{y}(\xi)\bigr)\,\mathbf m_b(\mathrm{d}\xi)\ge p\right\}.
\end{equation}
\par
 We are ready to prove (i). Define
\begin{equation}\label{eq:Mp}
  M_p:=\lambda\inf\left\{
      \sum_{i=1}^b t_i^\alpha:
      t_i\ge0,\quad
      \frac1b\sum_{i=1}^b g_A(\sigma t_i)\ge p
    \right\}.
\end{equation}
We first prove $Q_A(p)\geq M_p$. Let \(y\) be feasible for the closed-constraint problem \eqref{eq:positive-y}.
Since \(0<\alpha\le1\), it follows that for any $i=1,...,b$,
\begin{equation}\label{eq:subadditivity}
\Big(\sum_{v\succeq i}y_v\Big)^\alpha
  \le\sum_{v\succeq i}y_v^\alpha<\infty,
\end{equation}
where $v\succeq i$ means $v=i$ or $v$ is a descendant of $i$. This, combined with  $y_v\geq0$, yields that if $\xi|1=i$, then
\begin{equation}\label{eq:ray-bound}
  H_y(\xi)=\sum_{k\ge1}y_{\xi|k}\le \sum_{v\succeq i}y_v<\infty.
\end{equation}
This, combined with the fact that $\mathbf m_b([i])=\frac{1}{b}$, concludes that
\begin{align}\label{eq:compressed-feasible}
\frac1b\sum_{i=1}^b g_A\Big(\sigma t_i\Big)
  &=\frac1b\sum_{i=1}^b g_A\Big(\sigma \sum_{v\succeq i}y_v\Big)\cr
  &\ge\sum^{b}_{i=1}\int_{[i]}g_A\bigl(\sigma H_y(\xi)\bigr)\,
       \mathbf m_b(\mathrm{d}\xi)\cr
  &=\int_{\partial\mathbb{T}_b}g_A\bigl(\sigma H_y(\xi)\bigr)\,
       \mathbf m_b(\mathrm{d}\xi)\geq p,
\end{align}
where $t_i=\sum_{v\succeq i}y_v$, $i=1,...,b$.
It follows that
\begin{equation}\label{eq:lower-bound-i}
  M_p\leq\lambda\sum_{i=1}^bt_i^{\alpha}
  =\lambda\sum_{i=1}^b(\sum_{v\succeq i}y_v)^\alpha
  \le\lambda\sum_{v\ne\varnothing}y_v^\alpha
  =\mathcal{I_{\alpha}}(y).
\end{equation}
This, together with \eqref{eq:positive-y}, yields
\(M_p\leq Q_A(p)\). It remains to prove \(Q_A(p)\leq M_p\).
The feasible set in \eqref{eq:Mp} is nonempty and closed, and its intersection with
$\{(t_1,...,t_b):\sum_{i=1}^b t_i^\alpha\leq L,t_i\ge0\}$ is compact for any $L>0$. Hence the infimum in \eqref{eq:Mp} is attained. Let \((t_1^*,\ldots,t_b^*)\) minimize
\eqref{eq:Mp}, and set
\begin{equation}\label{eq:first-gen-field}
  y_i^*:=t_i^*\quad(\lvert i\rvert=1),
  \qquad
  y_v^*:=0\quad(\lvert v\rvert\ge2).
\end{equation}
Then
\begin{equation}\label{eq:upper-bound-i}
  \int_{\partial\mathbb{T}_b}g_A\bigl(\sigma H_{y^*}(\xi)\bigr)\,
  \mathbf m_b(\mathrm{d}\xi)
  =\frac1b\sum_{i=1}^bg_A(\sigma t_i^*)\ge p,
  \quad
  \mathcal{I_{\alpha}}(y^*)=\lambda\sum_{i=1}^b(t_i^*)^\alpha=M_p.
\end{equation}
This entails that \(Q_A(p)\leq M_p\), and proves \textup{(i)}.
\par
We proceed to prove \textup{(ii)}. Let $y$ be feasible for the closed-constraint problem in \eqref{eq:positive-y}. Observe that by concavity of $g_A(\sigma\cdot)$ on $[0,\infty)$ and Jensen's inequality,
\begin{equation}\label{eq:Jensen}
\begin{aligned}
  p\le \int_{\partial\mathbb{T}_b}g_A\left(\sigma H_y(\xi)\right)\,\mathbf m_b(\mathrm{d}\xi)\leq g_A\left(\sigma\int_{\partial\mathbb{T}_b}H_y(\xi)\,\mathbf m_b(\mathrm{d}\xi)\right).
\end{aligned}
\end{equation}
By using similar arguments in \eqref{eq:path-L1}, it follows that
\begin{equation}\label{eq:weighted-lower}
  \sum_{v\ne\varnothing}b^{-\lvert v\rvert}y_v
  =\int_{\partial\mathbb{T}_b}H_y(\xi)\,\mathbf m_b(\mathrm{d}\xi)
  \ge I_A(p),
\end{equation}
where the inequality follows from \eqref{eq:Jensen} and increasing property of \(g_A(\sigma\,\cdot)\). We first deal with
\(\alpha=1\) in (ii).  Since \(b^{-\lvert v\rvert}\le b^{-1}\), by \eqref{eq:weighted-lower}, it follows that
\begin{align}\label{67i89ou71q}
  Q_A(p)&=\inf\left\{\lambda\sum_{v\ne\varnothing}y_v:
    y\in\mathcal{H}_b,~y_v\geq0\ \text{for}\ {v\in\mathbb{T}_b},\ \int_{\partial\mathbb{T}_b}g_A\bigl(\sigma H_{y}(\xi)\bigr)\,\mathbf m_b(\mathrm{d}\xi)\ge p\right\}\cr
    &\geq\lambda bI_A(p).
\end{align}
Moreover, by setting \(y_v=I_A(p)\) for $|v|=1$ and \(y_v=0\) for $|v|\neq1$, we can show that
$$Q_A(p)\leq\lambda bI_A(p).$$
This, combined with \eqref{67i89ou71q}, concludes the case \(\alpha=1\) in \textup{(ii)}.

\par
We proceed to deal with \(\alpha>1\) in (ii). By H\"older's inequality,
\begin{align}\label{eq:Holder}
  I_A(p)
  \le\sum_{v\ne\varnothing}b^{-\lvert v\rvert}y_v&\le\left(\sum_{v\ne\varnothing}
       (b^{-\lvert v\rvert})^{\alpha/(\alpha-1)}\right)^{(\alpha-1)/\alpha}
       \left(\sum_{v\ne\varnothing}y_v^\alpha\right)^{1/\alpha}\cr
  &=\bigl(b^{1/(\alpha-1)}-1\bigr)^{-(\alpha-1)/\alpha}
       \left(\sum_{v\ne\varnothing}y_v^\alpha\right)^{1/\alpha}.
\end{align}
The two inequalities in \eqref{eq:Holder} can be attained by setting
\[
  \bigl(I_A(p)(1-r)r^{-1}\bigr)^\alpha
  \bigl(b^{-\lvert v\rvert}\bigr)^{\alpha/(\alpha-1)}=y_v^\alpha,
  \qquad r=b^{-1/(\alpha-1)}.
\]
Thus, we can set
\[
  h_v^*=\sigma I_A(p)(1-r)r^{k-1},
  \quad |v|=k,~k\geq1.
\]
For every \(\xi\in\partial\mathbb{T}_b\),
\[
  H_{h^*}(\xi)
  =\sigma I_A(p)(1-r)\sum_{k\geq1}r^{k-1}
  =\sigma I_A(p),
\]
and therefore \(\Psi_A(h^*)=g_A(\sigma I_A(p))=p\). Moreover,
\begin{align*}
  \mathcal{I_{\alpha}}(h^*)
  &=\lambda I_A(p)^\alpha(1-r)^\alpha
    \sum_{k\geq1}b^k r^{\alpha(k-1)}\\
  &=\lambda b(1-r)^{\alpha-1}I_A(p)^\alpha\\
  &=\lambda\bigl(b^{1/(\alpha-1)}-1\bigr)^{\alpha-1}
    I_A(p)^\alpha.
\end{align*}
Together with \eqref{eq:Holder}, this proves the case \(\alpha>1\) in assertion \textup{(ii)}.
\end{proof}
\begin{remark}
For the half-line \(A=(-\infty,a]\) and $\sigma=-1$,
the monotonicity condition in Theorem \ref{thm:dimension-reduction} holds. Moreover, if $a\geq 0$, then the concavity condition in Theorem \ref{thm:dimension-reduction} also holds.
\end{remark}
Next, we give a computable example for the constraint problem \eqref{eq:first-generation}.
Recall that $\Phi$ is the distribution function of a standard normal random variable.
\begin{example}\label{5tre6yy6yu}
Assume \(0<\alpha<1\). Let \(A=(-\infty,a]\) and \(\sigma=-1\). Choose \(p\in(\nu(A),1)\) such that
\begin{align}\label{43thy6yy6}
  s:=b\bigl(p-\Phi(a)\bigr)\in (0,1-\Phi(a)),
  \quad
    t^{*}(a+t^{*})\le 1-\alpha,
\end{align}
where \(t^{*}:=\Phi^{-1}(\Phi(a)+s)-a>0\). Then the minimum in
\eqref{eq:first-generation} is attained, up to
permutation, by
\[
  (t_1,\ldots,t_b)=(t^{*},0,\ldots,0),
\]
and hence
\begin{equation}\label{eq:halfline-computable-formula}
  Q_A(p)
  =\lambda\left[
    \Phi^{-1}\bigl(bp-(b-1)\Phi(a)\bigr)-a
  \right]^\alpha.
\end{equation}
\end{example}
\begin{proof}
Since
\(g_A(-t)=\Phi(a+t)\) for $A=(-\infty,a]$, by Theorem \ref{thm:dimension-reduction} (i), it remains to solve the following finite-dimensional minimization problem:
\begin{align}
  Q_A(p)
  &=\lambda\min\left\{
      \sum_{i=1}^b t_i^\alpha:
      t_i\ge0,\quad
      \frac1b\sum_{i=1}^b\Phi(a+t_i)\ge p
    \right\}\cr
  &=\lambda\min\left\{
      \sum_{i=1}^b t_i^\alpha:
      t_i\ge0,\quad
      \sum_{i=1}^b\bigl[\Phi(a+t_i)-\Phi(a)\bigr]= s
    \right\}\cr
     &=\lambda\min\left\{
      \sum^b_{i=1}\psi(z_i):
      1-\Phi(a)\ge z_i\ge0,\quad
      \sum_{i=1}^bz_i=s
    \right\},
  \label{eq:halfline-first-generation}
\end{align}
where in the last equality, we use
\[
  z_i:=\Phi(a+t_i)-\Phi(a),
  \qquad
  \psi(z):=\left[\Phi^{-1}\bigl(\Phi(a)+z\bigr)-a\right]^\alpha.
\]
Note that the second equality in \eqref{eq:halfline-first-generation} uses $\sum_{i=1}^b\bigl[\Phi(a+t_i)-\Phi(a)\bigr]= s$ instead of $\sum_{i=1}^b\bigl[\Phi(a+t_i)-\Phi(a)\bigr]\geq s$. The reasons are as follows. Since both \(t_i^\alpha\) and
\(\Phi(a+t_i)\) are increasing in \(t_i\), a minimizing vector must satisfy
\(\sum^b_{i=1}z_i=s\): if the inequality were strict, one could continuously
decrease a positive coordinate until equality holds, thereby decreasing the
objective function $\sum_{i=1}^b t_i^\alpha$.
\par
We proceed to minimize \(\sum^b_{i=1}\psi(z_i)\) over
 $z_i\in[0,1-\Phi(a)]$ and \(\sum^b_{i=1}z_i=s\). Writing
\(t=\Phi^{-1}(\Phi(a)+z)-a\), for \(0<z\le s\), we have
\[
  \psi''(z)
  =\frac{\alpha t^{\alpha-2}}{\varphi(a+t)^2}
    \bigl[(\alpha-1)+t(a+t)\bigr].
\]
The function \(t\mapsto t(a+t)\) is convex. Therefore, by \eqref{43thy6yy6}, for
\(0\le t\le t^{*}\),
\[
  t(a+t)
  \le\max\{0,t^{*}(a+t^{*})\}
  \le1-\alpha.
\]
Thus \(\psi''\le0\) on
\((0,s]\), and \(\psi\) is concave on \([0,s]\). This, combined with \(\psi(0)=0\), gives that whenever \(x,y\ge0\) and \(0<x+y\le s\),
\[
  \psi(x)\ge\frac{x}{x+y}\psi(x+y),
  \qquad
  \psi(y)\ge\frac{y}{x+y}\psi(x+y).
\]
Hence,
\[
  \psi(x)+\psi(y)\ge\psi(x+y).
\]
Repeatedly merging two coordinates now yields
\[
  \sum_{i=1}^b\psi(z_i)
  \ge\psi\left(\sum_{i=1}^b z_i\right)
  =\psi(s),
\]
where the inequality can be improved to an equality at \((z_1,\ldots,z_b)=(s,0,\ldots,0)\). This proves
\eqref{eq:halfline-computable-formula}.
\end{proof}
\section{Discussion}\label{rgrthrt45vg}
The common first-generation shift used by Chen and
He~\cite[Section~4.1.3]{ChenHe2019} assigns the displacement
\(\sigma I_A(p)\) to all \(b\) first-generation particles and has energy
\(\lambda bI_A(p)^\alpha\). Under the assumptions of
Theorem~\ref{thm:dimension-reduction}, this value is recovered when
\(\alpha=1\). For \(\alpha>1\), the proof of
Theorem~\ref{thm:dimension-reduction} identifies the minimizing field
\(h_v^*=\sigma I_A(p)(1-r)r^{|v|-1}\), $v\in\mathbb{T}_b^\circ$, where
\(r=b^{-1/(\alpha-1)}\). This gives
\(\displaystyle
Q_A(p)=\lambda\bigl(b^{1/(\alpha-1)}-1\bigr)^{\alpha-1}
I_A(p)^\alpha
\).
Since
\(\bigl(b^{1/(\alpha-1)}-1\bigr)^{\alpha-1}<b\) for $\alpha>1$, the common
first-generation shift is not optimal in this case. For
\(0<\alpha<1\), the half-line example leading to
\eqref{eq:halfline-computable-formula} has, up to permutation, the
minimizing vector \((t^*,0,\ldots,0)\). Thus only one first-generation
particle receives a nonzero displacement, which is inherited by its
descendants. These minimizers also describe the corresponding
large-deviation bounds. For the lower bound, one takes \(q>p\), uses a
finite-generation approximation of a minimizing field for \(Q_A(q)\),
and then lets \(q\downarrow p\). For the upper bound, the energy estimates
\eqref{eq:lower-bound-i} and \eqref{eq:Holder}, together with the merging
argument following \eqref{eq:halfline-first-generation}, show that no
feasible field has smaller tree energy.
\par
The deterministic assumption is substantive. When the maximal offspring
number \(B\) is strictly larger than the minimal offspring number \(b\),
the event can exploit both a displacement and an atypical population
imbalance. The resulting two-scale mechanism in
Chen and He~\cite[Theorem 1.3]{ChenHe2019} is not described by the present boundary
measure $\mathbf m_b$ or by simply replacing \(b\) with a random tree norm. Likewise, the
heavy-tail mechanisms studied by Zhang~\cite{Zhang2022} involve an
exceptional family or displacement and require different normalizations. The symmetry assumption for the step size was used to obtain the same logarithmic rate in both directions. If the two tails have constants \(\lambda_+\) and \(\lambda_-\), then the natural
candidate tree energy is \[
  \sum_{v\neq\varnothing}
  \Big\{
    \lambda_+(h_v^+)^\alpha+
    \lambda_-(h_v^-)^\alpha
  \Big\},
\]
where $x^-:=\max\{-x,0\}$.

\textbf{Acknowledgements.} Shuxiong Zhang is supported by NSF of China (No. 12501176), and by the National Key R\&D Program of
China 2022YFA1006102. Yaping Zhu is the corresponding author. The research of Yaping Zhu is supported by National Natural Science Foundation of China, Tianyuan Mathematics Young Researcher Project (No. 12526540) and the Fundamental Research Funds for the Central Universities.

\vspace{2cm}
	
	 \smallskip		
   \noindent{\bf Shuxiong Zhang:}  School of Mathematics and Statistics, Anhui Normal University, Wuhu, China.

   \noindent{\bf Email:} {\texttt
   	shuxiong.zhang@ahnu.edu.cn}

	\noindent{\bf Yaping Zhu:}  Department of Mathematics, Shanghai University of Finance and Economics, Shanghai, China.
	
	\noindent{\bf Email:} {\texttt
		zhuyaping@mail.sufe.edu.cn}	
\end{document}